\documentclass[10pt,reqno,final]{amsart}

\usepackage{amsmath,amssymb,amsthm}
\usepackage{mathrsfs}
\usepackage{pifont}
\usepackage[english]{babel}
\usepackage{graphicx}
\usepackage{subfig}
\usepackage{hyperref}
\usepackage{enumitem}
\usepackage[notcite,notref]{showkeys}
\usepackage{color}
\usepackage{tabularx}
\usepackage{booktabs}
\usepackage{array}
\usepackage{makecell}
\usepackage{multirow,multicol}
\usepackage{caption}
\usepackage{algorithm,algorithmicx}
\usepackage{algpseudocode}
\usepackage{fancybox}

\usepackage[all]{xy}

\usepackage[numbers,sort&compress]{natbib}

\allowdisplaybreaks

\theoremstyle{plain}
\newtheorem{theorem}{Theorem}[section]
\newtheorem{lemma}{Lemma}[section]
\newtheorem{corollary}{Corollary}[section]
\newtheorem{proposition}{Proposition}[section]

\theoremstyle{definition}
\newtheorem{definition}{Definition}[section]
\newtheorem{example}{Example}

\theoremstyle{remark}

\numberwithin{equation}{section}
\numberwithin{figure}{section}
\numberwithin{table}{section}

\graphicspath{{./figures/}}

\title[Block-Decomposition for 3-Parameter Persistence Modules]{Block-Decomposition for 3-Parameter Persistence Modules}

\author[Siheng Yi]{Siheng Yi}

\thanks{${}^{*}$Siheng Yi (E-mail: \texttt{12131237@mail.sustech.edu.cn})}
\thanks{${}^{}$Department of Mathematics, Southern University of Science and Technology}

\date{\today}



\begin{document}

\begin{abstract}
    In 2020, Cochoy and Oudot got the necessary and sufficient condition of the block-decomposition of 2-parameter persistence modules $\mathbb{R}^2 \to \textbf{Vec}_{\Bbbk}$. 
    And in 2024, Lebovici, Lerch and Oudot resolve the problem of block-decomposability for multi-parameter persistence modules. 
    Following the approach of Cochoy and Oudot's proof of block-decomposability for 2-parameter persistence modules, we rediscuss the necessary and sufficient conditions for the block decomposition of the 3-parameter persistence modules $\mathbb{R}^3 \to \textbf{Vec}_{\Bbbk}$. 
    Our most important contribution is to generalize the strong exactness of 2-parameter persistence modules to the case of 3-parameter persistence modules. 
    What's more, the generalized method allows us to understand why there is no block decomposition in general persistence modules to some extent. 
\end{abstract}

\maketitle

\section{Introduction}

1-parameter persistent homology had become an important method in topological data analysis\cite{carlsson2009topology,otter2017roadmap,rabadan2019topological} due to its special perspective, topological structures. 
As more and more researchers pay attention to this emerging feature extraction method, many scholars are considering using multi-parameter persistent homology to extract richer topological information from the data\cite{adcock2014classification,xia2015multidimensional,beltramo2021euler,biasotti2008multidimensional,carriere2020multiparameter}. 
The theoretical core of 1-parameter persistent homology is its decomposition theorem\cite{zomorodian2004computing} and the stability theorem\cite{cohen2005stability}, so for multi-parameter persistent homology, we also want to know if it has relative decomposition theorem and stability theorem. 
However, unlike 1-parameter persistent homology, there is currently no classification theorem for multi-parameter persistent homology, which is not completely characterized by some discrete descriptor.  
When abstracting persistent homology as persistence modules\cite{bubenik2014categorification}, the task of finding descriptors becomes a problem related to representation and algebra, so we hope to find some discrete complete invariants to persistence modules. 
Unfortunately, this cannot be achieved. Carlsson and Zomorodian\cite{carlsson2007theory} tell us that even a sufficiently simple class of 2-parameter persistence modules cannot be completely classified by discrete invariants. 
Therefore, some researchers have begun to consider the decomposition theorem of some special persistence modules. 
Among them, Cochoy and Oudot\cite{cochoy2020decomposition} proved that 2-parameter persistence modules $\mathbb{R}^2 \to \textbf{Vec}_{\Bbbk}$ can be decomposed into the direct sum of block modules if and only if 2-parameter persistence modules satisfy the 2-parameter strong exactness. 
Furthermore in 2024, Lebovici, Lerch, and Oudot resolve the problem of block-decomposability for multi-parameter persistence modules\cite{lebovici2024local}. 
In the paper\cite{lebovici2024local} by Lebovici et al., the authors introduced the middle exactness condition for $n$-parameter persistence modules and proved that this condition is both necessary and sufficient for block-decomposability of $n$-parameter persistence modules. 
In the special case of 3-parameter persistence modules, middle exactness is equivalent to our defined strong exactness condition.
Although Lebovici et al. addressed a more general scenario in their paper\cite{lebovici2024local}, which is the block-decomposition of any finite-dimensional persistence modules, and our work only solved the block-decomposition of 3-parameter persistence modules, I maintain that our work retains scholarly value. 
The 3-parameter strong exactness allows us to understand why there is no block decomposition in general persistence modules to some extent.

Our work can be seen as a continuation of the work of Cochoy and Oudot\cite{cochoy2020decomposition}. 
Following their proof approach, we have proved the necessary and sufficient conditions for $3$-dimensional persistence modules to be decomposed into the direct sum of block modules. 
The most important contribution of our work is to generalize 2-parameter strong exactness to the 3-dimensional case and prove that $3$-persistence modules can be decomposed as the direct sum of block modules if and only if the $3$-persistence modules satisfy $3$-dimensional strong exactness. 
Moreover, this generalization can actually be achieved in any finite-dimensional case. 
	
\begin{theorem}
    Let $\mathbb{M}$ be a pointwise finite-dimensional 3-parameter persistence module satisfying the 3-parameter strong exactness. 
    Then $\mathbb{M}$ may decompose uniquely (up to isomorphism and reordering of the terms) as a direct sum of block modules:
    $$\mathbb{M} \cong \underset{B:\text{blocks}}{\bigoplus}\mathbb{M}_B$$ in which $M_B\cong \bigoplus_{i=1}^{n_B}\Bbbk_{B}$ in which $n_B$ are determined by the counting functor $\mathcal{CF}$. 
\end{theorem}

\section{Preliminary}
To state our results, we need to define some notations. 
\subsection{The Definition of Persistence Modules}
In the paper, we consider $\mathbb{R}^3$ as a poset with the product order:
$$(x_1,x_2,x_3)\leq (y_1,y_2,y_3) \in \mathbb{R}^3 \Leftrightarrow x_i\leq y_i \text{ for all }i.$$
A persistence module indexed over $\mathbb{R}^3$ is a functor $\mathbb{M}:(\mathbb{R}^3,\leq) \to \textbf{Vec}_{\Bbbk}$ which is called an $3$-persistence module, where $\textbf{Vec}_{\Bbbk}$ is the category of finitely dimensional vector spaces over $\Bbbk$. 
For convenience of representation, we will denote $\mathbb{M}(\boldsymbol{t})$ with $\boldsymbol{t} \in \mathbb{R}^3$ as $\mathbb{M}_{\boldsymbol{}}$ and $\mathbb{M}(\boldsymbol{s} \leq \boldsymbol{t})$ with $\boldsymbol{s} \leq \boldsymbol{t} \in \mathbb{R}^3$ as $\rho_{\boldsymbol{s}}^{\boldsymbol{t}}$ which is a linear map. 
The persistence modules $\mathbb{M}$ we define are called pointwise finite-dimensional (\textbf{pfd}). 
	
\subsection{Cuts, Cuboids, and Block}
A cut on the real numbers $\mathbb{R}$ is a partition of $\mathbb{R}$ into two disjoint subsets $c^+$ and $c^-$ such that for every $x\in c^-$ and $y \in c^+$, the inequality $x<y$ holds. 
This definition formalizes the idea of "splitting" $\mathbb{R}$ into a lower set $c^-$ and an upper set $c^+$, where every element of $c^-$ lies strictly below every element of $c^+$. 
\vspace{0.3cm}
\begin{example}
    Showing two different cuts:
    \begin{itemize}
        \item $c=(c^-,c^+)$ with $c^-=(-\infty,1]$ and $c^+=(1,+\infty)$;
		\item $c=(c^-,c^+)$ with $c^-=(-\infty,1)$ and $c^+=[1,+\infty)$
    \end{itemize}
\end{example}
\vspace{0.3cm}
If $c^-=\emptyset$ or $c^+=\emptyset$, we call the cut $c$ trivial.
	
In $\mathbb{R}^3$, we can determine a cuboid $C$ by $3$ pairing cuts $(c_1,c^1), (c_2,c^2),(c_3,c^3)$ in which $c_1,c^1,c_2,c^2,c_3, c^3$ are cuts, so $C=({c_1}^+ \cap {c^1}^-)\times ({c_2}^+ \cap {c^2}^-) \times ({c_3}^+ \cap {c^3}^-)$. 
Every cuboid does not necessarily have to be open or closed.
Some special cuboids, called blocks, will be detailed in the following.
    
These blocks can be divided into three major classes: layer block, birth block, and death block. The first major class is further divided into $3$ sub-classes, each shown below.
Let $C=({c_1}^+ \cap {c^1}^-)\times ({c_2}^+ \cap {c^2}^-) \times ({c_3}^+ \cap {c^3}^-)$,
\begin{itemize}
    \item If all cuts except $c_i, c^i$ are trivial, we call $C$ a $i$-layer block; 
    \item If $c^1,c^2,c^3$ are trivial, we call $C$ a birth block; 
    \item If $c_1,c_2,c_3$ are trivial, we call $C$ a death block.
\end{itemize}

\begin{figure}[H]
    \centering
    \includegraphics[width=1\linewidth]{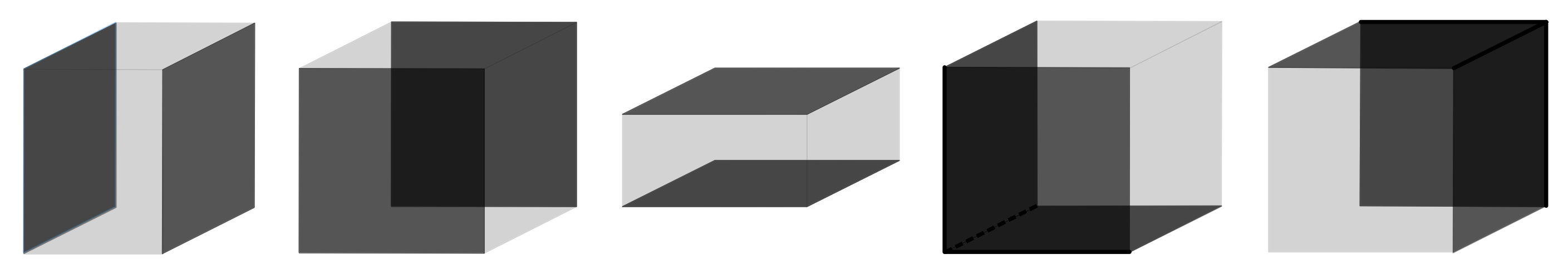}
    \caption{From left to right: three classes of layer blocks, birth blocks, death blocks}
    \label{fig:enter-label}
\end{figure}

\subsection{The Results of Cochoy and Oudot}
Before continuing the discussion, we need to review some of Cochoy and Oudot's definitions and results\cite{cochoy2020decomposition}. 
Cochoy and Oudot considered the block decomposition of 2-parameter persistence modules and proved the theorem of decomposition of \textbf{pfd} and strongly exact 2-parameter persistence modules.
	
In $\mathbb{R}^2$, we may also define 2-parameter cuboids, rectangles $R$, by two pairing cuts, $R=({c_1}^+ \cap {c^1}^-)\times ({c_2}^+ \cap {c^2}^-)$. 
What's more, the special rectangles, which are blocks, are as follows: 
\begin{figure}[H]
    \centering
    \includegraphics[width=1\linewidth]{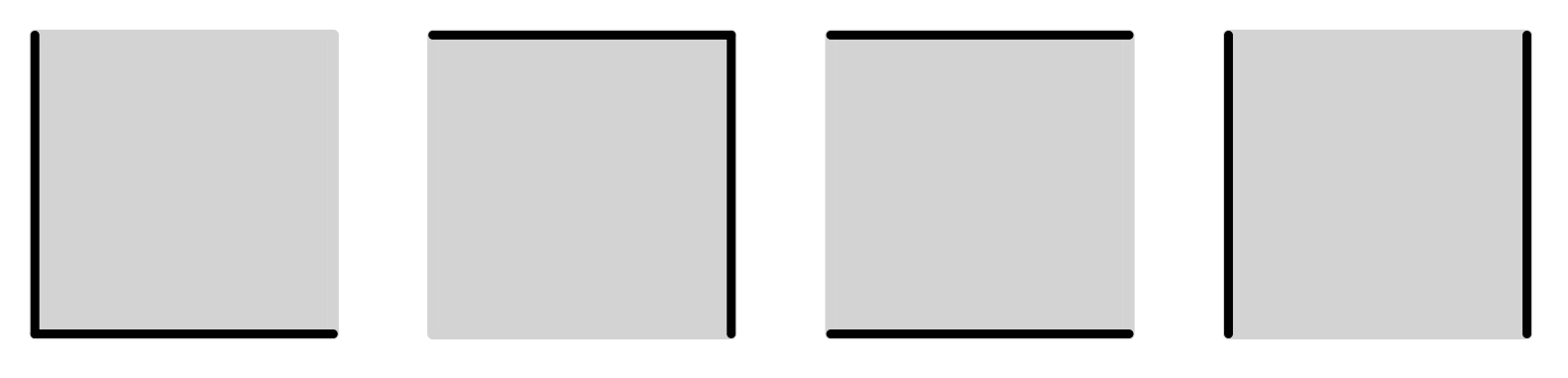}
    \caption{From left to right: birth blocks, death blocks, horizontal blocks, vertical blocks}
    \label{fig:enter-label}
\end{figure}

In 2-parameter persistence modules $\mathbb{M}: (\mathbb{R}^2,\leq) \to \textbf{Vec}_{\Bbbk}$, for any $(x_1,x_2)\leq (y_1,y_2) \in \mathbb{R}^2$, we have following commutative diagram:
$$
\xymatrix{
    \mathbb{M}_{(x_1,y_2)} \ar[r]^{\rho_{(x_1,y_2)}^{(y_1,y_2)}} & \mathbb{M}_{(y_1,y_2)}\\
    \mathbb{M}_{(x_1,x_2)} \ar[r]^{\rho_{(x_1,x_2)}^{(y_1,x_2)}} \ar[u]^{\rho_{(x_1,x_2)}^{(x_1,y_2)}} & \mathbb{M}_{(y_1,x_2)} \ar[u]_{\rho_{(y_1,x_2)}^{(y_1,y_2)}}\\
}$$
If for all $(x_1,x_2) \leq (y_1,y_2) \in \mathbb{R}^2$, the following sequence is exact, we call the 2-parameter persistence module $\mathbb{M}$ is \textbf{2-parameter strongly exact}. 
$$\mathbb{M}_{(x_1,x_2)} \xrightarrow[]{(\rho_{(x_1,x_2)}^{(x_1,y_2)},\rho_{(x_1,x_2)}^{(y_1,x_2)})} \mathbb{M}_{(x_1,y_2)}\oplus \mathbb{M}_{(y_1,x_2)} \xrightarrow[]{\rho_{(x_1,y_2)}^{(y_1,y_2)}-\rho_{(y_1,x_2)}^{(y_1,y_2)}} \mathbb{M}_{(y_1,y_2)}$$
\vspace{0.3cm}
\begin{theorem}
    Let $\mathbb{M}$ be a pointwise finite-dimensional and strongly exact 2-parameter persistence module. 
    Then, $\mathbb{M}$ decomposes uniquely (up to isomorphism and reordering of the terms) as a direct sum of block modules: 
    $$\mathbb{M}\cong \bigoplus_{B\in \mathcal{B}(\mathbb{M})} \Bbbk_B$$
    where $\Bbbk_{B}$ is the block module associated with a block $B$, and $\mathcal{B}(\mathbb{M})$ is a multiset of blocks determined by $\mathbb{M}$. 
\end{theorem}
\vspace{0.3cm}
For a block $B$, a block module $\Bbbk_B$ is defined as follows
\begin{equation}
    (\Bbbk_B)_{\boldsymbol{t}}=
    \begin{cases}
        \Bbbk,&\quad\boldsymbol{t}\in B\\
        0,&\quad\boldsymbol{t}\notin B\\
    \end{cases}
\end{equation}
and for any $\boldsymbol{s}\leq \boldsymbol{t}$, the morphisms $\rho_{\boldsymbol{s}}^{\boldsymbol{t}}$ in $\Bbbk_B$ are
\begin{equation}
    \rho_{\boldsymbol{s}}^{\boldsymbol{t}}=
    \begin{cases}
        id,&\quad\text{if }\boldsymbol{s},\boldsymbol{t}\in B\\
        0,&\quad\text{otherwise.}\\
    \end{cases}
\end{equation}

\section{Main Results and Their Proof}
Before we begin this section, it is necessary to explain that the results presented in this section were obtained by us at the end of 2023, and at that time, we chose not to make them public. 
However, in 2024, Lerch et al.\cite{lebovici2024local} published a more general solution to block-decomposability for multi-parameter persistence modules on arXiv. 
Despite this, I have decided to retain this content in my thesis because my proof method follows the approach\cite{cochoy2020decomposition} used by Oudot in solving the block-decomposability for 2-parameter persistence modules, and I believe this approach can be applied to the proof of the block decomposition theorem for multi-parameter persistence modules. 
Additionally, our perspective on the generalization of the high-dimensional exactness condition differs from Lerch's, which is why I believe it is meaningful to include this part.

To solve the block decomposition of 3-parameter persistence modules, we first need to generalize the strong exactness to a 3-parameter case. 

Let the following diagram be a commutative diagram in $\textbf{Vec}_{\Bbbk}$
$$\xymatrix{
    B\ar[r]^{g_1}&D\\
    A\ar[u]^{f_1} \ar[r]^{f_2}&C \ar[u]^{g_2}	
}$$
and deduce two commutative diagram
$$\xymatrix{
    &B\ar[r]^{g_1}&D\\
    &B\prod_{D}C\ar[u]\ar[r]&C\ar[u]^{g_2}\\
    A\ar@/^/[uur]^{f_1}\ar[ur]_{\exists !}^{f}\ar@/_/[urr]_{f_2}&&
}$$
    
$$\xymatrix{
    &&D\\
    B\ar@/^/[urr]^{g_1} \ar[r]&B\coprod_{A} C\ar[ur]_{\exists !}^{g}&\\
    A\ar[u]^{f_1}\ar[r]^{f_2}&C\ar@/_/[uur]_{g_2} \ar[u]&
}$$
in which $B\prod_{D}C=\{b+c \in B\oplus C : g_1(b)=g_2(c) \}$ and $B\coprod_{A} C = B\oplus C / \sim $, which $\sim$ is a deduced equivalent relation by $f_1(a)\sim f_2(a)$ for any $a\in A$. 
\vspace{0.3cm}
\begin{lemma}
    The following conditions are equivalence
    \begin{itemize}
		\item The sequence 
        $A \xrightarrow[]{(f_1,f_2)} B \oplus C \xrightarrow[]{g_1-g_2} D$
		is exact;
		\item $f$ is surjective;
		\item $g$ is injective.
    \end{itemize}
\end{lemma}
\begin{proof}
    (1)$\Rightarrow$(2):
    For any $(b,c)\in B\prod_{D}C$, we can find a vector $a\in A$ such that $f_1(a)=b$ and $f_2(a)=c$ due to the strong exactness. So $f(a)=(b,c)$, $f$ is surjective.
		
    (2)$\Rightarrow$(3):
    For any $b\in B$ and $c\in C$ such that $[b+c] \in B\coprod_{A} C$, if $g([b+c])=0$, then we have $g([b])=g([-c])$, that is $g_1(b)=g_2(-c)$. So $(b,-c)\in B\prod_{D}C$, and we can find out $a\in A$ such that $f_1(a)=b$ and $f_2(a)=-c$. So $[b+c]=0 \in B\coprod_{A} C$. $g$ is injective. 
		
    (3)$\Rightarrow$(1):
    For any $b \in B$ and $c \in C$ with $g_1(b)=g_2(c)$, $g([b-c])=g_1(b)-g_2(c)=0$. Since $g$ is injective, $[b]=[c]$. Thus there is a vector $a\in A$ such that $f_1(a)=b$ and $f_2(a)=c$.
\end{proof}
\vspace{0.3cm}
In the general case, we may also consider computing $f$ and $g$ similarly to the two-dimensional case.
Let $S$ be a finite set with $|S|=n$. The power set of $S$, $\mathcal{P}(S)=\{T: T\subseteq S\}$, is partially ordered set via inclusion. Let $\mathcal{P}_0(S)=\mathcal{P} \setminus \{\emptyset \}$ and $\mathcal{P}_1(S)=\mathcal{P}(S) \setminus \{S\}$.
A functor $\mathcal{X}:\mathcal{P}(S)\to \textbf{Vec}_{\Bbbk}$ is a commutative diagram shaped like a $n$-cube. 
What's more, we can get two morphisms $\psi:\mathcal{X}(\emptyset) \to \underset{T\in \mathcal{P}_0(S)}{\text{lim}}\mathcal{X} (T)$ and $\varphi:\underset{T\in \mathcal{P}_1(S)}{\text{colim}}\mathcal{X}(T) \to \mathcal{X}(S) $ naturally.

Consider 2-parameter persistence modules $M$, then any $(x_1,x_2)\leq (y_1,y_2) \in \mathbb{R}^2$, we can get a commutative diagram
$$
\xymatrix{
    \mathbb{M}_{(x_1,y_2)} \ar[r]^{\rho_{(x_1,y_2)}^{(y_1,y_2)}} & \mathbb{M}_{(y_1,y_2)}\\
    \mathbb{M}_{(x_1,x_2)} \ar[r]^{\rho_{(x_1,x_2)}^{(y_1,x_2)}} \ar[u]^{\rho_{(x_1,x_2)}^{(x_1,y_2)}} & \mathbb{M}_{(y_1,x_2)} \ar[u]_{\rho_{(y_1,x_2)}^{(y_1,y_2)}}\\
}$$
and the diagram deduces to a functor $\mathcal{X}:\mathcal{P}(S) \to \textbf{Vec}_{\Bbbk}$ with $|S|=2$. 
Thus for any functor $\mathcal{X}:\mathcal{P}(S) \to \textbf{Vec}_{\Bbbk}$ obtained by the above method, $\psi:\mathcal{X}(\emptyset) \to \underset{T\in \mathcal{P}_0(S)}{\text{lim}}\mathcal{X} (T)$ and $\varphi:\underset{T\in \mathcal{P}_1(S)}{\text{colim}}\mathcal{X}(T) \to \mathcal{X}(S)$ generated by the functor $\mathcal{X}$ are surjective and injective respectively if and only if $M$ is strongly exact. 
    
Now, we are considering block decomposition of $3$-dimensional persistence modules $\mathbb{M}:\mathbb{R}^3 \to \textbf{Vec}_{\Bbbk}$, so we need to extend the strong exactness about 2-parameter persistence modules to the conditions about $3$-dimensional persistence modules. Similar to the case of 2-parameter persistence modules, when considering the $3$-dimensional persistence modules, for any $(x_1,x_2,x_3) \leq (y_1,y_2,y_3) \in \mathbb{R}^3$, there is a commutative diagram like $3$-cube and the diagram induces the functor $\mathcal{X}(S):\mathcal{P}(S) \to \textbf{Vec}_{\Bbbk}$ with $|S|=3$, resulting in two morphisms $\psi:\mathcal{X}(\emptyset) \to \underset{T\in \mathcal{P}_0(S)}{\text{lim}}\mathcal{X} (T)$ and $\varphi:\underset{T\in \mathcal{P}_1(S)}{\text{colim}}\mathcal{X}(T) \to \mathcal{X}(S)$.
$$
\xymatrix{
    &\mathbb{M}_{(x_1,y_2,y_3)}\ar[rr]&&\mathbb{M}_{(y_1,y_2,y_3)}\\
    \mathbb{M}_{(x_1,x_2,y_3)}\ar[ru]\ar[rr]&&\mathbb{M}_{(y_1,x_2,y_3)}\ar[ru]&\\
    &\mathbb{M}_{(x_1,y_2,x_3)}\ar[uu]|\hole \ar[rr]|\hole &&\mathbb{M}_{(y_1,y_2,x_3)}\ar[uu]\\
    \mathbb{M}_{(x_1,x_2,x_3)}\ar[rr]\ar[uu]\ar[ru]&&\mathbb{M}_{(y_1,x_2,x_3)}\ar[uu]\ar[ru]&
}$$
Thus, when we consider the block decomposition of 3-parameter persistence modules, the strong exactness of 2-parameter block-decomposable persistence modules can be generalized to the following conditions: 3-parameter strong exactness. 
\vspace{0.3cm}
\begin{example}
    Consider the following 3-parameter persistence module $\mathbb{M}: \{ 0,1 \}^3 \to \textbf{Vec}_{\Bbbk}$
    $$
    \xymatrix{
		&\Bbbk\ar[rr]^{f} &&\Bbbk\oplus \Bbbk\\
		0\ar[ru]\ar[rr]&&\Bbbk\ar[ru]^{g}&\\
		&0\ar@{-->}[uu] \ar@{-->}[rr] &&\Bbbk\ar[uu]_{h}\\
		0\ar[rr]\ar[uu]\ar[ru]&&0\ar[uu]\ar[ru]&
    }
    $$
    where $f=\begin{pmatrix} 1\\ 0\end{pmatrix}$, $g=\begin{pmatrix} 0\\ 1\end{pmatrix}$, and $h=\begin{pmatrix} 1\\ 1\end{pmatrix}$. 
    It is obvious that the 3-parameter persistence module $\mathbb{M}$ is not block-decomposable. 
\end{example}
\vspace{0.3cm}
In the previous example, we know that only requiring $\varphi$ to be injective does not guarantee that 3-parameter persistence modules are block-decomposable. 
What's more, in 3-parameter cases, the two conditions that $\varphi$ is injective and $\psi$ is subjective are not equivalent.  
Therefore, it is reasonable to assume that $\varphi$ and $\psi$ are injective and subjective, respectively. 
\vspace{0.3cm}
\begin{definition}
    We say that 3-parameter persistence module $\mathbb{M}:\mathbb{R}^3 \to \textbf{Vec}_{\Bbbk}$ is 3-parameter strongly exact if the following conditions are satisfied
    \begin{itemize}
        \item for any $r \in \mathbb{R}$, $\mathbb{M}|_{\{r\} \times \mathbb{R} \times \mathbb{R}}$, $\mathbb{M}|_{\mathbb{R} \times \{r\} \times \mathbb{R}}$, $\mathbb{M}|_{\mathbb{R} \times \mathbb{R} \times \{r\}}$ are among 2-parameter strongly exact.
		\item for any $(x_1,x_2,x_3) \leq (y_1,y_2,y_3) \in \mathbb{R}^3$, the associated morphisms $\psi$ and $\varphi$ is surjective and injective respectively.
    \end{itemize} 
\end{definition}
\vspace{0.3cm}
We want to prove that if a 3-parameter persistence module $\mathbb{M}$ is strongly exact, then $\mathbb{M}$ can be decomposed as a direct sum of block modules. 
So we need to define the block modules and find all submodules of $\mathbb{M}$, which live exactly in blocks. 
\vspace{0.3cm}
\begin{definition}
    A persistence module $\mathbb{M}$ is called a block module if there is a block $\mathbb{M}$ such that $\mathbb{M}\cong \Bbbk_B$. 
\end{definition}
\vspace{0.3cm}

    \subsection{Some Basic Definitions and Results}
    In the 1-dimensional case, the interval modules of $\mathbb{M}:\mathbb{R} \to \textbf{Vec}_{\Bbbk}$ can be easily found since $\mathbb{R}$ is a totally ordered set. By computing $V_{I,t}^{+}:=\text{Im}_{I,t}^{+}\cap \text{Ker}_{I,t}^{+}$ and $V_{I,t}^{-}:=\text{Im}_{I,t}^{+}\cap \text{Ker}_{I,t}^{-}+\text{Im}_{I,t}^{-}\cap \text{Ker}_{I,t}^{+}$ in which $I\ni t$ is a interval of $\mathbb{R}$ and
    \begin{equation}
		\begin{aligned}
		  &\text{Im}_{I,t}^{+}={\underset{s\leq t}{\bigcap_{s\in I}}}\text{Im} \rho_s^t , &\text{Im}_{I,t}^{-}={\underset{s\leq t}{\sum_{s\notin I}}}\text{Im}\rho_s^t \\
		  &\text{Ker}_{I,t}^{+}={\underset{u\geq t}{\bigcap_{u\notin I}}}\text{Ker}\rho_t^u , &\text{Ker}_{I,t}^{-}={\underset{u\geq t}{\sum_{u\in I}}}\text{Ker}\rho_t^u\\
		\end{aligned}
    \end{equation}
    we can get $V_{I,t}^+/V_{I,t}^- \cong (\Bbbk_{I})_t$. 
    For $I=[a,b]$, $V_{I,t}^+/V_{I,t}^-$ denotes that the vector space whose dimension equals the number of generators, which were born at $a$ and died at $b$. 

    However $\mathbb{R}^n$, for $n\geq 2$, is not the totally ordered set, which results in $\text{Im}_{B,t}^{-}\subseteq \text{Im}_{B,t}^{+}$ and $\text{Ker}_{B,t}^{-}\subseteq \text{Ker}_{B,t}^{+}$, which hold in $1$-dimensional persistence modules, not holding in high-dimensional persistence modules, in which $B \ni t$ is any block.
    Thus we need to redefine $\text{Im}_{B,t}^{\pm}$ and $\text{Ker}_{B,t}^{\pm}$ in which $B \subseteq \mathbb{R}^3$ is any block and $t \in B$ .
    
    Firstly, we can establish the following notation in any persistence modules $\mathbb{M}: P\to \textbf{Vec}_{\Bbbk}$ in which $P$ is a poset: 
    \begin{equation}
		\begin{aligned}
		  &I_{P,t}^{+}:={\underset{s\leq t}{\bigcap_{s\in P}}}\text{Im }\rho_s^t , &I_{P,t}^{-}:={\underset{s\leq t}{\sum_{s\notin P}}}\text{Im }\rho_s^t \\
		  &K_{P,t}^{+}:={\underset{u\geq t}{\bigcap_{u\notin P}}}\text{Ker }\rho_t^u , &K_{P,t}^{-}:={\underset{u\geq t}{\sum_{u\in P}}}\text{Ker }\rho_t^u\\
		\end{aligned}
    \end{equation}
    But we know that $I_{P,t}^{+}\nsubseteq I_{P,t}^{+}$ and $K_{P,t}^{-} \nsubseteq K_{P,t}^{+}$ from the above discussion.
    Thus, We define that 
    \begin{equation}
		\begin{aligned}
		  &\text{Im}_{P,t}^{+}:=I_{P,t}^{+}, &\text{Im}_{P,t}^{-}:=I_{P,t}^{-} \cap I_{P,t}^{+} ,\\
		  &\text{Ker}_{P,t}^{+}:=K_{P,t}^{+}+K_{P,t}^{-}, &\text{Ker}_{P,t}^{-}:=K_{P,t}^{-}.
		\end{aligned}
    \end{equation}
    Obviously, $\text{Im}_{C,\boldsymbol{t}}^{-} \subset \text{Im}_{C,\boldsymbol{t}}^{+} $ and $\text{Ker}_{C,\boldsymbol{t}}^{-} \subset \text{Ker}_{C,\boldsymbol{t}}^{+}$.
    
    When we consider the \textbf{pfd} 3-parameter persistence module $\mathbb{M}:\mathbb{R}^3 \to \textbf{Vec}_{\Bbbk}$, the poset $P$ is a cuboid in $\mathbb{R}^3$, which is determined by three pairing cuts $\{c_1,c^1,c_2,c^2,c_3,c^3\}$ that is $C=({c_1}^+ \cap {c^1}^-)\times ({c_2}^+\cap {c^2}^-)\times ({c_3}^+\cap {c^3}^-)$.

    For any $\boldsymbol{t}=(t_1,t_2,t_3) \in C$,
    We construct these limits by considering the restrictions of the module $\mathbb{M}$ along $x$-axis,$y$-axis, and $z$-axis, respectively
    \begin{equation}
		\begin{aligned}
            &\text{Im}_{c_1,\boldsymbol{t}}^{+}={\underset{x\leq t_1}{\bigcap_{x\in {c_1}^+}}}\text{Im } \rho_{(x,t_2,t_3)}^{\boldsymbol{t}} & \text{Im}_{c_1,\boldsymbol{t}}^{-}={\sum_{x\in {c_1}^-}}\text{Im } \rho_{(x,t_2,t_3)}^{\boldsymbol{t}}\\
            &\text{Im}_{c_2,\boldsymbol{t}}^{+}={\underset{x\leq t_2}{\bigcap_{x\in {c_2}^+}}}\text{Im } \rho_{(t_1,x,t_3)}^{\boldsymbol{t}} & \text{Im}_{c_2,\boldsymbol{t}}^{-}={\sum_{x\in {c_2}^-}}\text{Im } \rho_{(t_1,x,t_3)}^{\boldsymbol{t}}\\
            &\text{Im}_{c_3,\boldsymbol{t}}^{+}={\underset{x\leq t_3}{\bigcap_{x\in {c_3}^+}}}\text{Im } \rho_{(t_1,t_2,x)}^{\boldsymbol{t}} & \text{Im}_{c_3,\boldsymbol{t}}^{-}={\sum_{x\in {c_3}^-}}\text{Im } \rho_{(t_1,t_2,x)}^{\boldsymbol{t}}\\
            &\text{Ker}_{c^1,\boldsymbol{t}}^{+}={\bigcap_{x\in {c^1}^+}}\text{Ker } \rho_{\boldsymbol{t}}^{(x,t_2,t_3)} & \text{Ker}_{c^1,\boldsymbol{t}}^{-}={\underset{x\geq t_1}{\sum_{x\in {c^1}^-}}}\text{Ker } \rho_{\boldsymbol{t}}^{(x,t_2,t_3)}\\
            &\text{Ker}_{c^2,\boldsymbol{t}}^{+}={\bigcap_{x\in {c^2}^+}}\text{Ker } \rho_{\boldsymbol{t}}^{(t_1,x,t_3)} & \text{Ker}_{c^2,\boldsymbol{t}}^{-}={\underset{x\geq t_2}{\sum_{x\in {c^2}^-}}}\text{Ker } \rho_{\boldsymbol{t}}^{(t_1,x,t_3)}\\
            &\text{Ker}_{c^3,\boldsymbol{t}}^{+}={\bigcap_{x\in {c^3}^+}}\text{Ker } \rho_{\boldsymbol{t}}^{(t_1,t_2,x)} & \text{Ker}_{c^3,\boldsymbol{t}}^{-}={\underset{x\geq t_3}{\sum_{x\in {c^3}^-}}}\text{Ker } \rho_{\boldsymbol{t}}^{(t_1,t_2,x)}\\
		\end{aligned}
    \end{equation}
    Through simple computation, we can get 
    \begin{equation}
		\begin{aligned}
		  \text{Im}_{C,\boldsymbol{t}}^{+}&=\text{Im}_{c_1,\boldsymbol{t}}^{+}\cap \text{Im}_{c_2,\boldsymbol{t}}^{+} \cap \text{Im}_{c_3,\boldsymbol{t}}^{+}\\
		  \text{Im}_{C,\boldsymbol{t}}^{-}&=\text{Im}_{c_1,\boldsymbol{t}}^{-}\cap \text{Im}_{c_2,\boldsymbol{t}}^{+} \cap \text{Im}_{c_3,\boldsymbol{t}}^{+} + \text{Im}_{c_1,\boldsymbol{t}}^{+}\cap \text{Im}_{c_2,\boldsymbol{t}}^{-} \cap \text{Im}_{c_3,\boldsymbol{t}}^{+} + \text{Im}_{c_1,\boldsymbol{t}}^{+}\cap \text{Im}_{c_2,\boldsymbol{t}}^{+} \cap \text{Im}_{c_3,\boldsymbol{t}}^{-}\\
          \text{Ker}_{C,\boldsymbol{t}}^{+}&=\text{Ker}_{c^1,\boldsymbol{t}}^{-}+\text{Ker}_{c^2,\boldsymbol{t}}^{-}+\text{Ker}_{c^3,\boldsymbol{t}}^{-}+\text{Ker}_{c^1,\boldsymbol{t}}^{+}\cap \text{Ker}_{c^2,\boldsymbol{t}}^{+}\cap \text{Ker}_{c^3,\boldsymbol{t}}^{+}\\
		  \text{Ker}_{C,\boldsymbol{t}}^{-}&=\text{Ker}_{c^1,\boldsymbol{t}}^{-}+\text{Ker}_{c^2,\boldsymbol{t}}^{-}+\text{Ker}_{c^3,\boldsymbol{t}}^{-}
		\end{aligned}
    \end{equation}
    \textbf{Note:} If we do not make any special explanation, all the persistence modules we will discuss later are \textbf{pfd} 3-parameter persistence modules $\mathbb{M}:\mathbb{R}^3\to \textbf{Vec}_{\Bbbk}$ satisfying the 3-parameter strong exactness. 
	
    The following lemma allows these concepts, such as $\text{Im}_{c_1,\boldsymbol{t}}^{\pm}$, $\text{Ker}_{c^1,\boldsymbol{t}}^{\pm}$, involving infinity to be discussed concretely
    \vspace{0.3cm}
    \begin{lemma}\label{1}
		$\mathbb{M}$ can be extended to the persistence module over $[-\infty,+\infty]^3$ by defining $\mathbb{M}_{(\infty,\cdot,\cdot)}=\mathbb{M}_{(\cdot,\infty,\cdot)}=\mathbb{M}_{(\cdot,\cdot,\infty)}=0$. Then 
		$$
		\text{Im}_{c_1,\boldsymbol{t}}^{+}=\text{Im }\rho_{(x,t_2,t_3)}^{\boldsymbol{t}} \text{ for some } x\in {c_1}^+\cap (-\infty,t_1] \text{ and any lower } x\in {c_1}^+,
		$$
		$$
		\text{Im}_{c_1,\boldsymbol{t}}^{-}=\text{Im }\rho_{(x,t_2,t_3)}^{\boldsymbol{t}} \text{ for some } x\in {c_1}^-\cup \{\infty\} \text{ and any greater } x\in {c_1}^-,
		$$
		$$
		\text{Ker}_{c^1,\boldsymbol{t}}^{+}=\text{Ker }\rho_{\boldsymbol{t}}^{(x,t_2,t_3)} \text{ for some } x\in {c^1}^+ \cup \{+\infty\} \text{ and any lower } x\in {c^1}^+,
		$$
		$$
		\text{Ker}_{c^1,\boldsymbol{t}}^{-}=\text{Ker }\rho_{\boldsymbol{t}}^{(x,t_2,t_3)} \text{ for some } x\in {c^1}^- \cap [t_1,+\infty) \text{ and any greater } x \in {c^x}^-. 
		$$
    The results for the cuts $c_2,c^2,c_3,c^3$ are similar to those for $c_1,c^1$, so we will not elaborate on them further. 
    \end{lemma}
    \vspace{0.3cm}
    Due to the 3-parameter strong exactness, we can decompose the image and kernel in 3-parameter persistence module $\mathbb{M}:\mathbb{R}^3 \to \textbf{Vec}_{\Bbbk}$ into a simpler form. 
    \vspace{0.3cm}
    \begin{lemma}\label{3}
		For any $\boldsymbol{s}\leq \boldsymbol{t} \in \mathbb{R}^3$, we have
		
		$\text{Im }\rho_{\boldsymbol{s}}^{\boldsymbol{t}}=\text{Im }\rho_{(s_1,t_2,t_3)}^{\boldsymbol{t}}\cap \text{Im }\rho_{(t_1,s_2,t_3)}^{\boldsymbol{t}}\cap \text{Im }\rho_{(t_1,t_2,s_3)}^{\boldsymbol{t}},$
		
		$\text{Ker }\rho_{\boldsymbol{s}}^{\boldsymbol{t}}=\text{Ker }\rho_{\boldsymbol{s}}^{(t_1,s_2,s_3)}+\text{Ker }\rho_{s}^{(s_1,t_2,s_3)}+\text{Ker }\rho_{\boldsymbol{s}}^{(s_1,s_2,t_3)}.$
    \end{lemma}    
    \begin{proof}
		Let the following commutative diagram satisfy 3-parameter strong exactness.
		We only need to prove $\text{Im }\rho_{X}^{Y}=\text{Im }\rho_{A}^{Y}\cap \text{Im }\rho_{B}^{Y}\cap \text{Im }\rho_{C}^{Y}$ and $\text{Ker }\rho_{X}^{Y}=\text{Ker }\rho_{X}^{D}+\text{Ker }\rho_{X}^{E}+\text{Ker }\rho_{X}^{F}$. 
		$$
		\xymatrix{
			&A\ar[rr]&&Y\\
			D\ar[ru]\ar[rr]&&B\ar[ru]&\\
			&E\ar[uu]|\hole \ar[rr]|\hole &&C\ar[uu]\\
			X\ar[rr]\ar[uu]\ar[ru]&&F\ar[uu]\ar[ru]&
		}
		$$
		
		(1) Obviously, $\text{Im }\rho_{X}^{Y}\subseteq \text{Im }\rho_{A}^{Y}\cap \text{Im }\rho_{B}^{Y}\cap \text{Im }\rho_{C}^{Y}$. If $a\in A$, $b\in B$ and $c\in C$ such that $\rho_{A}^{Y}(a)=\rho_{B}^{Y}(b)=\rho_{C}^{Y}(c)=y$, we may find out $d\in D$, $e\in E$ and $f\in F$ by the 2-parameter strong exactness. 
        Note that the construction of $\underset{T\in \mathcal{P}_0(S)}{\text{lim}}\mathcal{X} (T)$ and $\psi:\mathcal{X}(\emptyset) \to \underset{T\in \mathcal{P}_0(S)}{\text{lim}}\mathcal{X} (T)$ is surjective, we can find out $x\in X$ so that $\rho_{X}^{A}(x)=a$, $\rho_{X}^{B}(x)=b$ and $\rho_{X}^{C}(x)=c$. 
        Thus $\text{Im }\rho_{A}^{Y}\cap \text{Im }\rho_{B}^{Y}\cap \text{Im }\rho_{C}^{Y} \subseteq \text{Im }\rho_{X}^{Y}$. $\text{Im }\rho_{X}^{Y} = \text{Im }\rho_{A}^{Y}\cap \text{Im }\rho_{B}^{Y}\cap \text{Im }\rho_{C}^{Y}$.
		
		(2) We can directly obtain that $\text{Ker }\rho_{X}^{D}+\text{Ker }\rho_{X}^{E}+\text{Ker }\rho_{X}^{F} \subseteq \text{Ker }\rho_{X}^{Y}$.
		Let $x\in \text{Ker }\rho_{X}^{Y}$, and the image of $x$ at $A,B,C,D,E,F$ be $a,b,c,d,e,f$ respectively. 
		$$
		\xymatrix{
			&a\ar[rr]&&0\\
			d\ar[ru]\ar[rr]&&b\ar[ru]&\\
			&e\ar[uu]|\hole \ar[rr]|\hole &&c\ar[uu]\\
			x\ar[rr]\ar[uu]\ar[ru]&&f\ar[uu]\ar[ru]&
		}
		$$
		By the 2-parameter strong exactness, we can find out $\tilde{f} \in F$ and $\tilde{e}\in E$ such that $\rho_{E}^{A}(\tilde{e})=a$, $\rho_{E}^{C}(\tilde{e})=0$ and $\rho_{F}^{B}(\tilde{f})=b$, $\rho_{F}^{C}(\tilde{f})=0$. Then $\exists \tilde{x} \in X$ satisfies $\rho_{X}^{D}(\tilde{x})=d$, $\rho_{X}^{E}(\tilde{x})=\tilde{e}$ and $\rho_{X}^{F}(\tilde{x})=\tilde{f}$. So $x-\tilde{x} \in Ker\rho_{X}^{D}$. 
		$$
		\xymatrix{
			&a\ar[rr]&&0\\
			d\ar[ru]\ar[rr]&&b\ar[ru]&\\
			&\textcolor{red}{\tilde{e}}\ar[uu]|\hole \ar[rr]|\hole &&0\ar[uu]\\
			\textcolor{red}{\tilde{x}}\ar[rr]\ar[uu]\ar[ru]&&\textcolor{red}{\tilde{f}}\ar[uu]\ar[ru]&
		}
		$$
		Similarly, we can find out $\hat{d} \in D$ and $\hat{x}\in X$ by $0 \in A$, $b \in B$, $0 \in E$, $0 \in C$ and $\tilde{f} \in F$. 
        So $\hat{x}\in \text{Ker }\rho_{X}^{E}$. 
		$$
		\xymatrix{
			&0\ar[rr]&&0\\
			\textcolor{red}{\hat{d}}\ar[ru]\ar[rr]&&b\ar[ru]&\\
			&0\ar[uu]|\hole \ar[rr]|\hole &&0\ar[uu]\\
			\textcolor{red}{\hat{x}}\ar[rr]\ar[uu]\ar[ru]&&\tilde{f}\ar[uu]\ar[ru]&
		}
		$$
		Finally, we can easily prove that $\tilde{x}-\hat{x}\in \text{Ker } \rho_{X}^{F}$. 
		
		Thus $x=(x-\tilde{x})+\hat{x}+(\tilde{x}-\hat{x})$.
		So $\text{Ker }\rho_{X}^{Y} \subseteq \text{Ker }\rho_{X}^{D}+\text{Ker }\rho_{X}^{E}+\text{Ker }\rho_{X}^{F}$, then $\text{Ker }\rho_{X}^{Y} = \text{Ker }\rho_{X}^{D}+\text{Ker }\rho_{X}^{E}+\text{Ker }\rho_{X}^{F}$.
    \end{proof}
    \vspace{0.3cm}
    With the decomposition of $\text{Im }\rho_{\boldsymbol{s}}^{\boldsymbol{t}}$ and $\text{Ker }\rho_{\boldsymbol{s}}^{\boldsymbol{t}}$ in the Lemma\ref{3}, we can obtain the following crucial properties, which play an important role in finding submodules of $M$, which are block modules. 
    \vspace{0.3cm}
    \begin{lemma}\label{2}
		Let $\boldsymbol{s}\leq \boldsymbol{t} \in \mathbb{R}^3$ and $\clubsuit ,\spadesuit ,\bigstar \in \{+,-\}$. Then 
		
		$\rho_{\boldsymbol{s}}^{\boldsymbol{t}}(\text{Im}_{c_1,\boldsymbol{s}}^{\clubsuit}\cap \text{Im}_{c_2,\boldsymbol{s}}^{\spadesuit}\cap \text{Im}_{c_3,\boldsymbol{s}}^{\bigstar})=\text{Im}_{c_1,\boldsymbol{t}}^{\clubsuit}\cap \text{Im}_{c_2,\boldsymbol{t}}^{\spadesuit}\cap \text{Im}_{c_3,\boldsymbol{t}}^{\bigstar}$, 
		
		${(\rho_{\boldsymbol{s}}^{\boldsymbol{t}})}^{-1}(\text{Ker}_{c^1,\boldsymbol{t}}^{\clubsuit} + \text{Ker}_{c^2,\boldsymbol{t}}^{\spadesuit} + \text{Ker}_{c^3,\boldsymbol{t}}^{\bigstar})=\text{Ker}_{c^1,\boldsymbol{s}}^{\clubsuit} + \text{Ker}_{c^2,\boldsymbol{s}}^{\spadesuit} + \text{Ker}_{c^3,\boldsymbol{s}}^{\bigstar}$. 
    \end{lemma}
    \begin{proof}
		(1) The Lemme\ref{1} tells us that there exist $x\leq s_1 \leq t_1$, $y\leq s_2 \leq t_2$ and $z\leq s_3 \leq t_3$ (possibly equal to $-\infty$) such that 
		\begin{equation*}
			\begin{aligned}
				\text{Im}_{c_1,\boldsymbol{s}}^{\clubsuit}=\text{Im }\rho_{(x,s_2,s_3)}^{s} \text{ and } \text{Im}_{c_1,\boldsymbol{t}}^{\clubsuit}=\text{Im }\rho_{(x,t_2,t_3)}^{\boldsymbol{t}}\\
				\text{Im}_{c_2,\boldsymbol{s}}^{\spadesuit}=\text{Im }\rho_{(s_1,y,s_3)}^{\boldsymbol{s}} \text{ and } \text{Im}_{c_2,\boldsymbol{t}}^{\spadesuit}=\text{Im }\rho_{(t_1,y,t_3)}^{\boldsymbol{t}}\\
				\text{Im}_{c_3,\boldsymbol{s}}^{\bigstar}=\text{Im }\rho_{(s_1,s_2,z)}^{\boldsymbol{s}} \text{ and } \text{Im}_{c_3,\boldsymbol{t}}^{\bigstar}=\text{Im }\rho_{(t_1,t_2,z)}^{\boldsymbol{t}}
			\end{aligned}
		\end{equation*}
		Then, we can directly compute 
		\begin{equation*}
			\begin{aligned}
				\text{Im}_{c_1,\boldsymbol{t}}^{\clubsuit}\cap \text{Im }_{c_2,\boldsymbol{t}}^{\beta}\cap \text{Im}_{c_3,\boldsymbol{t}}^{\bigstar} &=\text{Im }\rho_{(x,t_2,t_3)}^{\boldsymbol{t}} \cap \text{Im}\rho_{(t_1,y,t_3)}^{\boldsymbol{t}} \cap \text{Im }\rho_{(t_1,t_2,z)}^{\boldsymbol{t}} \\
				&= \text{Im }\rho_{(x,y,z)}^{\boldsymbol{t}} = \rho_{\boldsymbol{s}}^{\boldsymbol{t}}(\text{Im }\rho_{(x,y,z)}^{\boldsymbol{s}})\\ &=\rho_{\boldsymbol{s}}^{\boldsymbol{t}}(\text{Im }\rho_{(x,s_2,s_3)}^{s} \cap \text{Im }\rho_{(s_1,y,s_3)}^{\boldsymbol{s}} \cap \text{Im }\rho_{(s_1,s_2,z)}^{\boldsymbol{s}})\\
				&= \rho_{\boldsymbol{s}}^{\boldsymbol{t}}(\text{Im}_{c_1,\boldsymbol{s}}^{\alpha}\cap \text{Im}_{c_2,\boldsymbol{s}}^{\beta}\cap \text{Im}_{c_3,\boldsymbol{s}}^{\bigstar})
			\end{aligned}
		\end{equation*}
		(2) Similar to (1), we can find out $(x,y,z) \geq \boldsymbol{t} \geq \boldsymbol{s} \in \mathbb{R}^3$ (possibly $x,y,z$ equal to $+\infty$) such that 
		\begin{equation*}
			\begin{aligned}
				\text{Ker}_{c^1,\boldsymbol{s}}^{\clubsuit}=\text{Ker }\rho_{\boldsymbol{s}}^{(x,s_2,s_3)} \text{ and } \text{Ker}_{c^1,\boldsymbol{t}}^{\clubsuit}=\text{Ker }\rho_{\boldsymbol{t}}^{(x,t_2,t_3)}\\
				\text{Ker}_{c^2,\boldsymbol{s}}^{\beta}=\text{Ker}\rho_{\boldsymbol{s}}^{(s_1,y,s_3)} \text{ and } \text{Ker}_{c^2,\boldsymbol{t}}^{\beta}=\text{Ker }\rho_{\boldsymbol{t}}^{(t_1,y,t_3)}\\
				\text{Ker}_{c^3,\boldsymbol{s}}^{\bigstar}=\text{Ker }\rho_{\boldsymbol{s}}^{(s_1,s_2,z)} \text{ and } \text{Ker}_{c^3,\boldsymbol{t}}^{\bigstar}=\text{Ker }\rho_{\boldsymbol{t}}^{(t_1,t_2,z)}
			\end{aligned}
		\end{equation*}
		Then 
		\begin{equation*}
			\begin{aligned}
				&{(\rho_{\boldsymbol{s}}^{\boldsymbol{t}})}^{-1}(\text{Ker }\rho_{\boldsymbol{t}}^{(x,t_2,t_3)}+\text{Ker }\rho_{\boldsymbol{t}}^{(t_1,y,t_3)}+\text{Ker }\rho_{\boldsymbol{t}}^{(t_1,t_2,z)})\\
				=&{(\rho_{\boldsymbol{s}}^{\boldsymbol{t}})}^{-1}(\rho_{\boldsymbol{t}}^{(x,y,z)})=\text{Ker } \rho_{\boldsymbol{s}}^{(x,y,z)}\\
                =&\text{Ker }\rho_{\boldsymbol{s}}^{(x,s_2,s_3)}+\text{Ker }\rho_{s}^{(s_1,y,s_3)}+\text{Ker }\rho_{\boldsymbol{s}}^{(s_1,s_2,z)}
			\end{aligned}
		\end{equation*}
    \end{proof}
    \vspace{0.3cm}
    From the above Lemma, we may easily prove the following result. 
    \vspace{0.3cm}
    \begin{corollary}
		$$
		\rho_{\boldsymbol{s}}^{\boldsymbol{t}}(\text{Im}_{C,\boldsymbol{s}}^{\pm})=\text{Im}_{C,\boldsymbol{t}}^{\pm} \text{ and } {(\rho_{\boldsymbol{s}}^{\boldsymbol{t}})}^{-1}(\text{Ker}_{C,\boldsymbol{t}}^{\pm})=\text{Ker}_{C,\boldsymbol{s}}^{\pm}
		$$
    \end{corollary}
    \begin{proof}
		We only need to pay attention to the facts that $f(U+V)=f(U)+f(V)$ and $f^{-1}(U\cap V)=f^{-1}(U) \cap f^{-1}(V)$. 
    \end{proof}
    \vspace{0.3cm}
    The following lemma states the relation between $\text{Ker}$ and $\text{Im}$. 
    \vspace{0.3cm}
    \begin{lemma}\label{7}
		If a 3-parameter persistence module $\mathbb{M}$ is \textbf{pfd} and satisfies the 3-parameter strong exactness, then
		$$\text{Ker}_{c^1,\boldsymbol{t}}^{\clubsuit}\subseteq \text{Im}_{c_2,\boldsymbol{t}}^{\spadesuit} \cap \text{Im}_{c_3,\boldsymbol{t}}^{\bigstar}$$
		in which
		\begin{itemize}
			\item if ${c^1}^+ \neq \emptyset$, then $\clubsuit=+$, else $\clubsuit=-$;
			\item if ${c_2}^- \neq \emptyset$, then $\spadesuit=-$, else $\spadesuit=+$;
			\item if ${c_3}^- \neq \emptyset$, then $\bigstar=-$, else $\bigstar=+$.
		\end{itemize} 
		
		Similarly, we have 
		$\text{Ker}_{c^2,\boldsymbol{t}}^{\clubsuit}\subseteq \text{Im}_{c_1,\boldsymbol{t}}^{\spadesuit} \cap \text{Im}_{c_3,\boldsymbol{t}}^{\bigstar}$ and $\text{Ker}_{c^3,\boldsymbol{t}}^{\clubsuit}\subseteq \text{Im}_{c_1,\boldsymbol{t}}^{\spadesuit} \cap \text{Im}_{c_2,\boldsymbol{t}}^{\bigstar}$. 
    \end{lemma}
    \begin{proof}
		We only prove $\text{Ker}_{c^1,\boldsymbol{t}}^{-}\subseteq \text{Im}_{c_2,\boldsymbol{t}}^{+}\cap \text{Im}_{c_3,\boldsymbol{t}}^{+}$, others are similar to $\text{Ker}_{c^1,\boldsymbol{t}}^{-}\subseteq \text{Im}_{c_2,\boldsymbol{t}}^{+}\cap \text{Im}_{c_3,\boldsymbol{t}}^{+}$.
		
		Let $\boldsymbol{t}=(t_1,t_2,t_3) \in \mathbb{R}^3$.
		From the Lemma\ref{1}, we can find out $x\in {c^1}^-$ and $y\in {c_2}^+$ such that $\text{Ker}_{c^1,\boldsymbol{t}}^{-}=\text{Ker } \rho_{\boldsymbol{t}}^{(x,t_2,t_3)}$ and $\text{Im}_{c_2,\boldsymbol{t}}^{+}=\text{Im } \rho_{(t_1,y,t_3)}^{\boldsymbol{t}}$. 
		Because $\mathbb{M}|_{\mathbb{R}\times\mathbb{R} \times \{t_3\}}$ satisfies 2-parameter strong exactness, we may consider following commutative diagram
		$$
		\xymatrix{
			\mathbb{M}_{(t_1,t_2,t_3)} \ar[r] & \mathbb{M}_{(x,t_2,t_3)}\\
			\mathbb{M}_{(t_1,y,t_3)} \ar[r] \ar[u] & \mathbb{M}_{(x,y,t_3)} \ar[u]\\
		}
		$$
		For any $\alpha \in \text{Ker}_{c^1,\boldsymbol{t}}^{-}$, we can find out a common antecedent $\beta \in \mathbb{M}_{(t_1,y,t_3)}$ with $0 \in \mathbb{M}_{(x,y,t_3)}$, then $\alpha \in \text{Im}_{c_2,\boldsymbol{t}}^{+}$. So $\text{Ker}_{c^1,\boldsymbol{t}}^{-} \subseteq \text{Im}_{c_2,\boldsymbol{t}}^{+}$. 
		
		Similarly, we easily prove $\text{Ker}_{c^1,\boldsymbol{t}}^{-} \subseteq \text{Im}_{c_3,\boldsymbol{t}}^{+}$. Therefore, $\text{Ker}_{c^1,\boldsymbol{t}}^{-}\subseteq \text{Im}_{c_2,\boldsymbol{t}}^{+}\cap \text{Im}_{c_3,\boldsymbol{t}}^{+}$. 
    \end{proof}
    \vspace{0.3cm}

    \subsection{Find Block Submodules in $\mathbb{M}$}
    Next, we will try to find out all the block modules $\Bbbk_B$, which are submodules of $\mathbb{M}$. 
    If $\mathbb{M}$ may be decomposed as a direct sum of block modules $\Bbbk_B$, the block modules $\Bbbk_B$ are exactly the submodules of $\mathbb{M}$, which the elements are born at the birth boundary of $B$ and die at death boundary of $B$. 

    For any cuboid $C=({c_1}^+ \cap {c^1}^-)\times({c_2}^+ \cap {c^2}^-)\times({c_3}^+ \cap {c^3}^-)$, we define $V_{C,\boldsymbol{t}}^{+}=\text{Im}_{C,\boldsymbol{t}}^{+}\cap \text{Ker}_{C,\boldsymbol{t}}^{+}$ and $V_{C,\boldsymbol{t}}^{-}=\text{Im}_{C,\boldsymbol{t}}^{+}\cap \text{Ker}_{C,\boldsymbol{t}}^{-}+\text{Im}_{C,\boldsymbol{t}}^{-}\cap \text{Ker}_{C,\boldsymbol{t}}^{+}$. 
    Obviously, $V_{C,\boldsymbol{t}}^{-}\subseteq V_{C,\boldsymbol{t}}^{+}$.

    According to our supposition, $V_{C,\boldsymbol{t}}^{+}/V_{C,\boldsymbol{t}}^{-}$ is isomorphic to $(\Bbbk_C)_t$, and combined with the previous results, we can obtain the following lemma. 
    \vspace{0.3cm}
    \begin{lemma}
		If a 3-parameter persistence module $\mathbb{M}$ is \textbf{pfd} and satisfies the 3-parameter strong exactness, then 
		$\rho_{\boldsymbol{s}}^{\boldsymbol{t}}(V_{C,\boldsymbol{s}}^{\pm})=V_{C,\boldsymbol{t}}^{\pm}$ and the induced morphism $\bar{\rho_{\boldsymbol{s}}^{\boldsymbol{t}}}:V_{C,\boldsymbol{s}}^{+}/V_{C,\boldsymbol{s}}^{-} \to V_{C,\boldsymbol{t}}^{+}/V_{C,\boldsymbol{t}}^{+}$ is an isomorphism.
    \end{lemma}
    \begin{proof}
		We can easily prove that 
		\begin{equation*}
			\begin{aligned}
				\rho_{\boldsymbol{s}}^{\boldsymbol{t}}(V_{C,\boldsymbol{s}}^{+})&=\rho_{\boldsymbol{s}}^{\boldsymbol{t}}(\text{Im}_{C,\boldsymbol{s}}^{+}\cap \text{Ker}_{C,\boldsymbol{s}}^{+})\\
				&\subseteq \rho_{\boldsymbol{s}}^{\boldsymbol{t}}(\text{Im}_{C,\boldsymbol{s}}^{+}) \cap \rho_{\boldsymbol{s}}^{\boldsymbol{t}}(\text{Ker}_{C,\boldsymbol{s}}^{+})\\
				&= \text{Im}_{C,\boldsymbol{t}}^{+}\cap \text{Ker}_{C,\boldsymbol{t}}^{+}=V_{C,\boldsymbol{t}}^{+},
			\end{aligned}
		\end{equation*}
		\begin{equation*}
			\begin{aligned}
				\rho_{\boldsymbol{s}}^{\boldsymbol{t}}(V_{C,\boldsymbol{s}}^{-})&=\rho_{\boldsymbol{s}}^{\boldsymbol{t}}(\text{Im}_{C,\boldsymbol{s}}^{+}\cap \text{Ker}_{C,\boldsymbol{s}}^{-}+\text{Im}_{C,\boldsymbol{s}}^{-}\cap \text{Ker}_{C,\boldsymbol{s}}^{+})\\
				&\subseteq \rho_{\boldsymbol{s}}^{\boldsymbol{t}}(\text{Im}_{C,\boldsymbol{s}}^{+})\cap\rho_{\boldsymbol{s}}^{\boldsymbol{t}}(\text{Ker}_{C,\boldsymbol{s}}^{-}) + \rho_{\boldsymbol{s}}^{\boldsymbol{t}}(\text{Im}_{C,\boldsymbol{s}}^{-})\cap \rho_{\boldsymbol{s}}^{\boldsymbol{t}}(\text{Ker}_{C,\boldsymbol{s}}^{+})\\
				&= \text{Im}_{C,\boldsymbol{t}}^{+}\cap \text{Ker}_{C,\boldsymbol{t}}^{-} + \text{Im}_{C,\boldsymbol{t}}^{-}\cap \text{Ker}_{C,\boldsymbol{s}}^{+}=V_{C,\boldsymbol{t}}^{-}.
			\end{aligned}
		\end{equation*}
		So we have $\rho_{\boldsymbol{s}}^{\boldsymbol{t}}(V_{C,\boldsymbol{s}}^{\pm})\subseteq V_{C,\boldsymbol{t}}^{\pm}$ and $\bar{\rho_{\boldsymbol{s}}^{\boldsymbol{t}}}:V_{R,\boldsymbol{s}}^{+}/V_{R,\boldsymbol{s}}^{-} \to V_{R,\boldsymbol{t}}^{+}/V_{R,\boldsymbol{t}}^{+}$. 
		Sequently, we prove that $\bar{\rho_{\boldsymbol{s}}^{\boldsymbol{t}}}$ is injective and surjective. 
		
		Surjectivity: For any $\beta \in V_{C,\boldsymbol{t}}^{+}=\text{Im}_{C,\boldsymbol{t}}^{+}\cap \text{Ker}_{C,\boldsymbol{t}}^{+}$, we can find out $\alpha \in \text{Im}_{C,\boldsymbol{s}}^{+}$ such that $\beta=\rho_{\boldsymbol{s}}^{\boldsymbol{t}}(\alpha)$. 
        Note that $\alpha \in {(\rho_{\boldsymbol{s}}^{\boldsymbol{t}})}^{-1}(\beta) \subseteq {(\rho_{\boldsymbol{s}}^{\boldsymbol{t}})}^{-1}(\text{Ker}_{C,\boldsymbol{t}}^{+})= \text{Ker}_{C,\boldsymbol{s}}^{+}$, then $\alpha \in V_{C,\boldsymbol{s}}^{+}$. 
        Thus $\bar{\rho_{\boldsymbol{s}}^{\boldsymbol{t}}}$ is surjective.
		
		Injectivity: Let $\beta =\rho_{\boldsymbol{s}}^{\boldsymbol{t}}(\alpha) \in V_{R,\boldsymbol{t}}^{-}$ in which $\alpha \in V_{C,\boldsymbol{s}}^{+}$. We have $\beta=\beta_1+\beta_2$ with $\beta_1 \in \text{Im}_{C,\boldsymbol{t}}^{-}\cap \text{Ker}_{C,\boldsymbol{t}}^{+}$ and $\beta_2 \in \text{Im}_{C,\boldsymbol{t}}^{+}\cap \text{Ker}_{C,\boldsymbol{t}}^{-}$. 
		By the same argument as before, $\beta_1=\rho_{\boldsymbol{s}}^{\boldsymbol{t}}(\alpha_1)$ for some $\alpha_1 \in \text{Im}_{C,\boldsymbol{s}}^{-}\cap \text{Ker}_{C,\boldsymbol{s}}^{+}$. 
		Because $\rho_{\boldsymbol{s}}^{\boldsymbol{t}}(\alpha-\alpha_1)=\beta_2 \in \text{Ker}_{C,\boldsymbol{t}}^{-}$, then $\alpha - \alpha_1 \in \text{Ker}_{C,\boldsymbol{s}}^{-}$. 
		Note that $\alpha, \alpha_1 \in \text{Im}_{C,\boldsymbol{s}}^{+}$, then $\alpha-\alpha_1 \in \text{Im}_{C,\boldsymbol{s}}^{+}$. 
        So $\alpha \in V_{C,\boldsymbol{s}}^{-}$ and $\bar{\rho_{\boldsymbol{s}}^{\boldsymbol{t}}}$ is injective. 
		It implies that $\bar{\rho_{\boldsymbol{s}}^{\boldsymbol{t}}}(V_{C,\boldsymbol{s}}^{-})=V_{C,\boldsymbol{t}}^{-}$. 
    \end{proof}
    \vspace{0.3cm}
    The above lemma tells us that we can find out the vectors that exactly live in the cuboid $C$ by using $V_{C,t}^{+}/V_{C,t}^{-}$. 
    However, we do not want to the vector space $V_{C,\boldsymbol{t}}^{+}/V_{C,\boldsymbol{t}}^{-}$ to depend on the selection of position $\boldsymbol{t}$. 
    So we need to define $\mathcal{CF}_C(\mathbb{M}):=\underset{\boldsymbol{t}\in C}{\underset{\longleftarrow}{\text{lim}}} V_{C,\boldsymbol{t}}^{+}/V_{C,\boldsymbol{t}}^{-}$.     
    The counting functor $\mathcal{CF}$ plays a central role in the decomposition of persistence modules. 
    Specifically, it is an additive functor\cite{cochoy2020decomposition} that determines the multiplicity of the summand $\Bbbk_C$ in the decomposition of the module $\mathbb{M}$ into a direct sum. 
    \vspace{0.3cm}
    \begin{lemma}
        Let $\mathbb{M}$ be \textbf{pfd} and decompose into a direct sum of cuboid modules. 
        For any cuboid $C$, the dimension of the vector space $\mathcal{CF}_C(\mathbb{M})$ precisely equals the multiplicity of the summand $\Bbbk_C$ in the decomposition of $\mathbb{M}$ into a direct sum. 
    \end{lemma}
    \begin{proof}
		Because $\mathcal{CF}$ is an additive functor, the proof can be reduced to demonstrating the result for a single summand $\Bbbk_{C'}$. 
		Suppose $C=({c_1}^+ \cap {c^1}^-) \times ({c_2}^+ \cap {c^2}^-) \times ({c_3}^+ \cap {c^3}^-)$, $C'=({c_{1'}}^+ \cap {c^{1'}}^-) \times ({c_{2'}}^+ \cap {c^{2'}}^-) \times ({c_{3'}}^+ \cap {c^{3'}}^-)$, and $C\neq C'$.
		We can find out a cut is different between $C$ and $C'$. Without loss of generality, let $c_1 \neq c_{1'}$. 
		For any $\boldsymbol{t}\in C\cap C'$, $\text{Im}_{c_1,t}^{+}(\Bbbk_{C'})=\text{Im}_{c_1,t}^{-}(\Bbbk_{C'})$, then $\text{Im}_{C,\boldsymbol{t}}^{-}(\Bbbk_{C'})=\text{Im}_{C,\boldsymbol{t}}^{+}(\Bbbk_{C'})$. 
        Thus $V_{C,\boldsymbol{t}}^{-}(\Bbbk_{C'})=V_{C,\boldsymbol{t}}^{+}(\Bbbk_{C'})$, that is $V_{C,\boldsymbol{t}}^{+}(\Bbbk_{C'})/ V_{C,\boldsymbol{t}}^{-}(\Bbbk_{C'})=0$. 
		What's more, for any $\boldsymbol{t} \in C-C'$, we have $(\Bbbk_{C'})_{\boldsymbol{t}}=0$, then $V_{C,\boldsymbol{t}}^{+}(\Bbbk_{C'})/ V_{C,\boldsymbol{t}}^{-}(\Bbbk_{C'})=0$. 
		So $\mathcal{CF}_C(\Bbbk_{C'})=0$. 
		If we consider that $c^1 \neq c^{1'}$, then we may get the same result by computing $\text{Ker}_{C,\boldsymbol{t}}^{\pm}(\Bbbk_{C'})$. 
		
		Secondly, we suppose that $C=C'$. For any $\boldsymbol{t} \in C$, we easily get $V_{C,\boldsymbol{t}}^{+}=\Bbbk$ and $V_{C,\boldsymbol{t}}^{-}=0$, thus $\mathcal{CF}_C(\Bbbk_{C'})=\Bbbk$. 
    \end{proof}
    \vspace{0.3cm}
    In order to obtain the submodule $\mathbb{M}_B$, which is a submodule of $\mathbb{M}$ and exactly distribute in the block $B$, we need to define $V_{B}^{\pm}(\mathbb{M}):=\underset{\boldsymbol{t} \in B}{\underset{\longleftarrow}{\text{lim}}} V_{B,\boldsymbol{t}}^{\pm}$ that is independent of the selection of position $\boldsymbol{t}$. 
    \vspace{0.3cm}
    \begin{lemma}\label{16}
        If the 3-parameter persistence module $\mathbb{M}$ is \textbf{pfd} and satisfies 3-parameter strong exactness, then $\mathcal{CF}_B(\mathbb{M})\cong V_{B}^{+}/V_{B}^{-}$.
    \end{lemma}
    \begin{proof}
        Obviously, $V_{B,\boldsymbol{s}}^{-} \subseteq V_{B,\boldsymbol{t}}^{-}$ for all $\boldsymbol{s} \leq \boldsymbol{t} \in B$.
		We easily know that $\{V_{B,\boldsymbol{s}}^{-},\rho_{\boldsymbol{s}}^{\boldsymbol{t}}\}_{\boldsymbol{s}\leq \boldsymbol{t} \in B}$ is an inverse system and the Mittag-Leffler condition holds for the inverse system since every space $V_{B,\boldsymbol{s}}^{-}$ is finite-dimensional. 
		Meanwhile, $B$ contains a countable subset that is coinitial for the product order $\leq$, and the collection of sequences is exact
		$$0 \to V_{B,\boldsymbol{t}}^{-} \to V_{B,\boldsymbol{t}}^{+} \to V_{B,\boldsymbol{t}}^{+}/V_{B,\boldsymbol{t}}^{-} \to 0.$$
		Thus, the limit sequence 
		$$0 \to V_{B}^{-}(\mathbb{M}) \to V_{B}^{+}(\mathbb{M}) \to \mathcal{CF}_B(\mathbb{M}) \to 0$$
		is exact by Proposition 13.2.2 of the reference\cite{grothendieck1961elements}. 
    \end{proof}
    \vspace{0.3cm}
    Let $\pi_{\boldsymbol{t}}:V_{B}^{+}(\mathbb{M}) \to V_{B,\boldsymbol{t}}^{+}$ denote the natural morphism induced by the universal property of $\underset{\boldsymbol{t} \in B}{\underset{\longleftarrow}{\text{lim}}} V_{B,\boldsymbol{t}}^{+}$. 
    Then we may get $V_{B}^{-}(\mathbb{M})=\underset{\boldsymbol{t}\in B}{\bigcap}\pi_{\boldsymbol{t}}^{-1}(V_{B,\boldsymbol{t}}^{-})$ and $V_{B}^{+}(\mathbb{M})=\underset{\boldsymbol{t}\in B}{\bigcap}\pi_{\boldsymbol{t}}^{-1}(V_{B,{\boldsymbol{t}}}^{+})$. Thus we have $V_{B}^{-}(\mathbb{M})\subset V_{B}^{+}(\mathbb{M})$. 
    \vspace{0.3cm}
    \begin{lemma}
        If a 3-parameter persistence module $\mathbb{M}$ is \textbf{pfd} and satisfies the 3-parameter strong exactness, then
		$$\bar{\pi}_{\boldsymbol{t}}:V_{B}^{+}(\mathbb{M})/V_{B}^{-}(\mathbb{M}) \to V_{B,\boldsymbol{t}}^{+}/V_{B,\boldsymbol{t}}^{-}$$
		is isomorphic. 
    \end{lemma}
    \begin{proof}
		Referring to the proof of Lemma 5.2 of \cite{cochoy2020decomposition}. 
    \end{proof}
    \vspace{0.3cm}
    After obtaining the previous results, we can select the appropriate subspace $M_B^0$ from $V_{B}^+$, so that the submodules $\mathbb{M}_B$ are generated through $\pi_{\boldsymbol{t}}(M_B^0)$. 
    \vspace{0.3cm}
    \begin{proposition}
		If the 3-parameter persistence module $\mathbb{M}$ is \textbf{pfd} and satisfies 3-parameter strong exactness, then the subspace $V_{B}^{-}$ has a complementary space $M_{B}^{0}$ in $V_{B}^{+}$ such that the following persistence module 
        \begin{equation}
            (\mathbb{M}_B)_{\boldsymbol{t}}=
            \begin{cases}
                \pi_{\boldsymbol{t}}(M_{B}^{0}),&\boldsymbol{t}\in B\\
                0,&\boldsymbol{t}\notin B\\
            \end{cases}
        \end{equation}
		is a submodule $\mathbb{M}_B$ of $\mathbb{M}$. 
    \end{proposition}
    \begin{proof}
		We will discuss the proof of the results in three cases: birth blocks, death blocks, and layer blocks. 
		For a fixed block $B$, regardless of the choice of subspace $M_{B}^{0}$ satisfying the decomposition $V_{B}^{+}(\mathbb{M})=M_{B}^{0} \oplus V_{B}^{-}(\mathbb{M})$, the following statements will hold: 
		\begin{itemize}
		  \item for any $\boldsymbol{s}, \boldsymbol{t} \in B$ satisfying $\boldsymbol{s} \leq \boldsymbol{t}$, $\rho_{\boldsymbol{s}}^{\boldsymbol{t}}((\mathbb{M}_B)_{\boldsymbol{t}}) \subseteq (\mathbb{M}_B)_{\boldsymbol{t}}$, since $\rho_{\boldsymbol{s}}^{\boldsymbol{t}} \circ \pi_{\boldsymbol{s}}=\pi_{\boldsymbol{t}}$ by the definition of $\pi$. 
		  \item for any $\boldsymbol{s} \notin B, \boldsymbol{t} \in B$ satisfying $\boldsymbol{s} \leq \boldsymbol{t}$, $\rho_{\boldsymbol{s}}^{\boldsymbol{t}}((\mathbb{M}_B)_{\boldsymbol{s}})=\rho_{\boldsymbol{s}}^{\boldsymbol{t}}(0)=0 \subseteq (\mathbb{M}_B)_{\boldsymbol{t}}$. 
		\end{itemize}
		There only remains to show that, for any $\boldsymbol{s} \leq \boldsymbol{t}$, $\boldsymbol{s}\in B$ and $\boldsymbol{t} \notin B$, $\rho_{\boldsymbol{s}}^{\boldsymbol{t}}((\mathbb{M}_B)_{\boldsymbol{s}})=0$. Therefore, we need to choose a suitable subspace $M_B^0$ that satisfies the condition.
		
		\textbf{Case $B$ is birth block}: $C=({c_1}^+ \cap {c^1}^-) \times ({c_2}^+ \cap {c^2}^-) \times ({c_3}^+ \cap {c^3}^-)$ in which ${c^1}^+ = {c^2}^+ = {c^3}^+ = \emptyset$. 
		
		For any choice of subspace $M_B^0$, the condition can be satisfied.
		
		\textbf{Case $B$ is death block}: $C=({c_1}^+ \cap {c^1}^-) \times ({c_2}^+ \cap {c^2}^-) \times ({c_3}^+ \cap {c^3}^-)$ in which ${c_1}^- = {c_2}^- = {c_3}^- = \emptyset$. 
		
		Let $K_{B,\boldsymbol{s}}^{+}=\text{Ker}_{c^1,\boldsymbol{s}}^{+}\cap \text{Ker}_{c^2,\boldsymbol{s}}^{+} \cap \text{Ker}_{c^3,\boldsymbol{s}}^{+}$ for all $\boldsymbol{s} \in B$. 
        The collection of these vector spaces, combined with the transition maps $\rho_{\boldsymbol{s}}^{\boldsymbol{t}}$ for $\boldsymbol{s}\leq \boldsymbol{t} \in B$ forms an inverse system. 
		Because $K_{B,\boldsymbol{s}}^{+} \subseteq \text{Im}_{B,\boldsymbol{s}}^{+}$ by Lemma\ref{2} and $K_{B,\boldsymbol{s}}^{+} \subseteq \text{Ker}_{B,\boldsymbol{s}}^{+}$, then $K_{B,\boldsymbol{s}}^{+}\subseteq V_{B,\boldsymbol{s}}^{+}$. Thus
		$$K_{B}^{+}(\mathbb{M})=\underset{\boldsymbol{s}\in B}{\underset{\longleftarrow}{\text{lim}}} K_{B,\boldsymbol{t}}^{+} = \underset{\boldsymbol{s}\in B}{\bigcap} \pi_{\boldsymbol{s}}^{-1}(K_{B,\boldsymbol{s}}^{+}) \subseteq V_{B}^{+}(\mathbb{M}).$$
		And for any $\boldsymbol{s}\in B$, following equation holds:
		\begin{equation*}
		  \begin{aligned}
		      V_{B,\boldsymbol{s}}^{+}&=\text{Im}_{B,\boldsymbol{s}}^{+}\cap \text{Ker}_{B,\boldsymbol{s}}^{+}=\text{Im}_{B,\boldsymbol{s}}^{+} \cap (\text{Ker}_{B,\boldsymbol{s}}^{-}+K_{B,\boldsymbol{s}}^{+})\\
				&=\text{Im}_{B,\boldsymbol{s}}^{+}\cap \text{Ker}_{B,\boldsymbol{s}}^{-} + \text{Im}_{B,\boldsymbol{s}}^{+} \cap K_{B,\boldsymbol{s}}^{+}=V_{B,\boldsymbol{s}}^{-}+K_{B,\boldsymbol{s}}^{+}.
		  \end{aligned}
		\end{equation*}
		In other words, for any $\boldsymbol{s}\in B$, the following sequence is exact: 
        $$0 \to V_{B,\boldsymbol{s}}^{-}\cap K_{B,\boldsymbol{s}}^{+} \xrightarrow[]{\alpha \mapsto (\alpha,-\alpha)} V_{B,\boldsymbol{s}}^{-}\oplus K_{B,\boldsymbol{s}}^{+} \xrightarrow[]{(\alpha,\beta) \mapsto \alpha+\beta} 
        V_{B,\boldsymbol{s}}^{+} \to 0 . $$
        This system of exact sequences satisfies the Mittag-Leffler condition since every space $V_{B,\boldsymbol{s}}^{-}\cap K_{B,\boldsymbol{s}}^{+}$ is finite-dimensional, and so, by Proposition13.2.2 of \cite{grothendieck1961elements}, the limit sequence is exact. 
		Note that $\underset{\boldsymbol{s}\in B}{\underset{\longleftarrow}{\text{lim}}} V_{B,\boldsymbol{s}}^{-}\cap K_{B,\boldsymbol{s}}^{+}=V_{B}^{-}(\mathbb{M}) \cap K_{B}^{+}(\mathbb{M})$ in $V_{B}^{+}(\mathbb{M})$, and the canonical morphism $V_{B}^{-}(\mathbb{M}) \oplus K_{B}^{+}(\mathbb{M}) \to \underset{\boldsymbol{s}\in B}{\underset{\longleftarrow}{\text{lim}}} V_{B,\boldsymbol{s}}^{-}\oplus K_{B,\boldsymbol{s}}^{+}$ is an isomorphism, then the following sequence is exact:
        $$ 0 \to V_{B}^{-}(\mathbb{M}) \cap K_{B}^{+}(\mathbb{M}) \xrightarrow[]{\alpha \mapsto (\alpha,-\alpha)} V_{B}^{-}(\mathbb{M})\oplus K_{B}^{+}(\mathbb{M}) \xrightarrow[]{(\alpha,\beta) \mapsto \alpha+\beta} V_{B}^{+}(\mathbb{M}) \to 0 $$
		which implies that $V_{B}^{-}(\mathbb{M})+K_{B}^{+}(\mathbb{M})=V_{B}^{+}(\mathbb{M})$. 
		Thus we only need to choose a complement subspace $M_B^0$ of $V_{B}^{-}(\mathbb{M})$ inside $K_{B}^{+}(\mathbb{M})$.
		
		\textbf{Case $B$ is strict layer block}:
		We only need to consider one case that is ${c_2}^-={c^2}^+={c_3}^-={c^3}^+=\emptyset$ and ${c_1}^-\neq \emptyset \neq {c^1}^+$, the rest are similar. 
		Let $K_{B,\boldsymbol{s}}^{+}=\text{Im}_{c_1,\boldsymbol{s}}^{+} \cap Ker_{c^1,\boldsymbol{s}}^{+}$ for any $\boldsymbol{s} \in B$.
		
		We have $V_{B,\boldsymbol{s}}^{+} = \text{Im}_{B,\boldsymbol{s}}^{+} \cap \text{Ker}_{B,\boldsymbol{s}}^{+} = \text{Im}_{c_1,\boldsymbol{s}}^{+} \cap \text{Im}_{c_2,\boldsymbol{s}}^{+} \cap \text{Im}_{c_3,\boldsymbol{s}}^{+} \cap (\text{Ker}_{c^1,\boldsymbol{s}}^{+}+\text{Ker}_{c^2,\boldsymbol{s}}^{-}+\text{Ker}_{c^3,\boldsymbol{s}}^{-})$. 
		Because $\text{Ker}_{c^2,\boldsymbol{s}}^{-} \subseteq \text{Im}_{c_1,\boldsymbol{s}}^{+} \cap \text{Im}_{c_3,\boldsymbol{s}}^{+}$, $\text{Ker}_{c^3,\boldsymbol{s}}^{-} \subseteq \text{Im}_{c_1,\boldsymbol{s}}^{+} \cap \text{Im}_{c_2,\boldsymbol{s}}^{+}$ and $\text{Ker}_{c^1,\boldsymbol{s}}^{-} \subseteq \text{Ker}_{c^1,\boldsymbol{s}}^{+}\subseteq \text{Im}_{c_2,\boldsymbol{s}}^{+} \cap \text{Im}_{c_3,\boldsymbol{s}}^{+}$ by Lemma\ref{2}, then we get 
		\begin{equation*}
		  \begin{aligned}
				V_{B,\boldsymbol{s}}^{+}&=\text{Im}_{c_1,\boldsymbol{s}}^{+} \cap \text{Im}_{c_2,\boldsymbol{s}}^{+} \cap \text{Im}_{c_3,\boldsymbol{s}}^{+} \cap (\text{Ker}_{c^1,\boldsymbol{s}}^{+}+\text{Ker}_{c^2,\boldsymbol{s}}^{-}+\text{Ker}_{c^3,\boldsymbol{s}}^{-})\\
				&= \text{Im}_{c_1,\boldsymbol{s}}^{+} \cap \text{Ker}_{c^1,\boldsymbol{s}}^{+} + \text{Im}_{c_2,\boldsymbol{s}}^{+} \cap \text{Ker}_{c^2,\boldsymbol{s}}^{-} + \text{Im}_{c_3,\boldsymbol{s}}^{+} \cap \text{Ker}_{c^3,\boldsymbol{s}}^{-}\\
				&= K_{B,\boldsymbol{s}}^{+} + \text{Im}_{c_2,\boldsymbol{s}}^{+} \cap \text{Ker}_{c^2,\boldsymbol{s}}^{-} + \text{Im}_{c_3,\boldsymbol{s}}^{+} \cap \text{Ker}_{c^3,\boldsymbol{s}}^{-}
		  \end{aligned}
		\end{equation*}
		And we have 
		\begin{equation*}
		  \begin{aligned}
				V_{B,\boldsymbol{s}}^{-}&=\text{Im}_{B,\boldsymbol{s}}^{+} \cap \text{Ker}_{B,\boldsymbol{s}}^{-} + \text{Im}_{B,\boldsymbol{s}}^{-} \cap \text{Ker}_{B,\boldsymbol{s}}^{+}\\
				&=\text{Im}_{c_1,\boldsymbol{s}}^{+} \cap \text{Ker}_{c^1,\boldsymbol{s}}^{-} + \text{Im}_{c_2,\boldsymbol{s}}^{+} \cap \text{Ker}_{c^2,\boldsymbol{s}}^{-} + \text{Im}_{c_3,\boldsymbol{s}}^{+} \cap \text{Ker}_{c^3,\boldsymbol{s}}^{-} + \text{Im}_{B,\boldsymbol{s}}^{-} \cap \text{Ker}_{B,\boldsymbol{s}}^{+}\\
		  \end{aligned}
		\end{equation*}
		Thus $V_{B,\boldsymbol{s}}^{+}=V_{B,\boldsymbol{s}}^{-}+K_{B,\boldsymbol{s}}^{+}$. 
        Following a similar argument to the preceding case, we conclude that the limits satisfy $V_{B}^{+}(\mathbb{M})=V_{B}^{-}(\mathbb{M})+K_{B}^{+}(\mathbb{M})$. 
        Therefore, we may select the vector space complement $M_{B}^{0}$ inside $K_{B}^{+}(\mathbb{M})$, guaranteeing that $\pi_{\boldsymbol{s}}(M_{B}^{0})\subseteq K_{B,\boldsymbol{s}}^{+}$ for any $\boldsymbol{s} \in B$. 
    \end{proof}
    \vspace{0.3cm}
	
    Because $V_{B}^{+}(\mathbb{M})=V_{B}^{-}\oplus M_{B}^{0}$, and $\pi_{\boldsymbol{t}}$ and $\pi_{\boldsymbol{t}}|_{V_{B}^{-}}$ are isomophisms, $M_{B}^{0}\cong (\mathbb{M}_B)_{\boldsymbol{t}}$ and $V_{B,\boldsymbol{t}}^{+}\cong V_{B,\boldsymbol{t}}^{-}\oplus (\mathbb{M}_B)_{\boldsymbol{t}}$. 
    \vspace{0.3cm}
    \begin{corollary}\label{8}
		For every block $B$, $\mathbb{M}_B \cong \underset{\text{dim }\mathcal{CF}_B(\mathbb{M})}{\bigoplus}\Bbbk_B$.
    \end{corollary}
    \vspace{0.3cm}

    \subsection{The Direct Sum Decomposition}
    Before we prove the direct sum decomposition of $\mathbb{M}:\mathbb{R}^3 \to \textbf{Vec}_{\Bbbk}$, which satisfies 3-parameter strong exactness, we need to introduce some definitions and results of disjointness and covering of sections\cite{crawley2015decomposition}. 

    In a vector space $U$, a section consists of two subspaces $(F^-,F^+)$ such that $F^- \subset F^+ \subset U$. 
    First, we introduce the notations of disjointness of sections. 
    We call that a collection of sections $\{(F_{\lambda}^{-},F_{\lambda}^{+})\}_{\lambda \in \Lambda}$ in $U$ is said to be disjoint, if whenever $\lambda \neq \mu$, one of the inclusions $F_{\lambda}^{+}\subseteq F_{\mu}^{-}$ or $F_{\mu}^{+}\subseteq F_{\lambda}^{-}$ is satisfied. 
    \vspace{0.3cm}
    \begin{lemma}{\cite{crawley2015decomposition}}\label{10}
		Let $\{(F_{\lambda}^{-},F_{\lambda}^{+})\}_{\lambda \in \Lambda}$ be a collection of sections in $U$, that is disjoint. 
        For any $\lambda\in \Lambda$, suppose that $V_{\lambda}$ is a subspace satisfying $F_{\lambda}^{+}=M_{\lambda}\oplus F_{\lambda}^{-}$, then $\sum V_{\lambda} \cong \bigoplus V_{\lambda}$. 
    \end{lemma}
    \vspace{0.3cm}
    \begin{lemma}{\cite{crawley2015decomposition}}\label{4}
		Given a fixed $\boldsymbol{t} \in \mathbb{R}^3$, each of the collections $\{(\text{Im}_{c_1,\boldsymbol{t}}^{-},\text{Im}_{c_1,\boldsymbol{t}}^{+})\}_{c_1:t_1 \in {c_1}^+}$, $\{(\text{Ker}_{c^1,\boldsymbol{t}}^{-},\text{Ker}_{c^1,\boldsymbol{t}}^{+})\}_{c^1:t_1 \in {c^1}^-}$, $\{(\text{Im}_{c_2,\boldsymbol{t}}^{-},\text{Im}_{c_2,\boldsymbol{t}}^{+})\}_{c_2:t_2 \in {c_2}^+}$, $\{(\text{Ker}_{c^2,\boldsymbol{t}}^{-},\text{Ker}_{c^2,\boldsymbol{t}}^{+})\}_{c^2:t_2 \in {c^2}^-}$, $\{(\text{Im}_{c_3,\boldsymbol{t}}^{-},\text{Im}_{c_3,\boldsymbol{t}}^{+})\}_{c_3:t_3 \in {c_3}^+}$ and $\{(\text{Ker}_{c^3,\boldsymbol{t}}^{-},\text{Ker}_{c^3,\boldsymbol{t}}^{+})\}_{c^3:t_3 \in {c^3}^-}$ is disjoint in $M_{\boldsymbol{t}}$.
    \end{lemma}
    \vspace{0.3cm}
    \begin{lemma}{\cite{crawley2015decomposition}}\label{9}
        Let the collection of sections in $U$, $\mathcal{F}=\{(F_{\lambda}^{-},F_{\lambda}^{+})\}_{\lambda \in \Lambda}$, be disjoint, and $\mathcal{G}=\{(G_{\sigma}^{-},G_{\sigma}^{+})\}_{\sigma \in \Sigma}$ be any collection of sections in $U$. 
        Then the collection of sections in $U$
		$$\{(F_{\lambda}^{-}+G_{\sigma}^{-}\cup F_{\lambda}^{+},F_{\lambda}^{-}+G_{\sigma}^{+}\cup F_{\lambda}^{+})\}_{(\lambda,\sigma)\in \Lambda \times \Sigma}$$
		is disjoint.
    \end{lemma}
    \vspace{0.3cm}
    Because $\mathcal{V}_{\boldsymbol{t}}:=\{(V_{B,\boldsymbol{t}}^{-}, V_{B,\boldsymbol{t}}^{+})\}_{B: block \ni \boldsymbol{t}}$ is not disjoint, we cannot directly study the direct sum decomposition of $\mathbb{M}$ by considering $\mathcal{V}_{\boldsymbol{t}}$. 
    Thus, we define the disjoint section $\mathcal{F}_{\boldsymbol{t}}:=\{F_{B,\boldsymbol{t}}^{-},F_{B,\boldsymbol{t}}^{+}\}_{B:block \ni \boldsymbol{t}}$
    
    \begin{equation*}
		\begin{aligned}
		  &F_{B,\boldsymbol{t}}^{+}=\text{Im}_{B,\boldsymbol{t}}^{-}+V_{B,\boldsymbol{t}}^{+}=\text{Im}_{B,\boldsymbol{t}}^{-}+\text{Ker}_{B,\boldsymbol{t}}^{+}\cap \text{Im}_{B,\boldsymbol{t}}^{+},\\
          &F_{B,\boldsymbol{t}}^{-}=\text{Im}_{B,\boldsymbol{t}}^{-}+V_{B,\boldsymbol{t}}^{-}=\text{Im}_{B,\boldsymbol{t}}^{-}+\text{Ker}_{B,\boldsymbol{t}}^{-}\cap \text{Im}_{B,\boldsymbol{t}}^{+}. 
		\end{aligned}
    \end{equation*}
    \vspace{0.3cm}
    \begin{lemma}
		$F_{B,\boldsymbol{t}}^{+}=F_{B,\boldsymbol{t}}^{-}\oplus(\mathbb{M}_B)_{\boldsymbol{t}}$
    \end{lemma}
    \begin{proof}
		From the definition of $F_{B,\boldsymbol{t}}^{\pm}$, we can easily know that $F_{B,\boldsymbol{t}}^{+}=F_{B,\boldsymbol{t}}^{-} + (\mathbb{M}_B)_{\boldsymbol{t}}$. 
		And, we have ${(\mathbb{M}_B)}_{\boldsymbol{t}} \subseteq V_{B,\boldsymbol{t}}^{+}$, so 
		$$F_{B,\boldsymbol{t}}^{-} \cap (\mathbb{M}_B)_{\boldsymbol{t}} = F_{B,\boldsymbol{t}}^{-} \cap V_{B,\boldsymbol{t}}^{+} \cap (\mathbb{M}_B)_{\boldsymbol{t}} = V_{B,\boldsymbol{t}}^{-}\cap (\mathbb{M}_B)_{\boldsymbol{t}} = 0$$
    \end{proof}
    \vspace{0.3cm}
    From this lemma, we see that we can study the direct sum of $\mathbb{M}$ by considering $\mathcal{F}_{\boldsymbol{t}}$. 

    We divide these type of blocks $B=({c_1}^+ \cap {c^1}^-)\times ({c_2}^+ \cap {c^2}^-) \times ({c_3}^+ \cap {c^3}^-)$ into following 5 types:
    \begin{itemize}
        \item $\mathcal{B}_1=\{ B| {c_2}^- = {c_2}^+ = {c_3}^- = {c^2}^+ = \emptyset \text{ and } {c^1}^+\neq \emptyset \}$; 
        \item $\mathcal{B}_2=\{ B| {c_1}^- = {c_1}^+ = {c_3}^- = {c^2}^+ = \emptyset \text{ and } {c^2}^+\neq \emptyset \}$; 
        \item $\mathcal{B}_3=\{ B| {c_1}^- = {c_1}^+ = {c_2}^- = {c^2}^+ = \emptyset \text{ and } {c^3}^+\neq \emptyset \}$; 
        \item $\mathcal{B}_4=\text{the set of all death blocks} \setminus \bigcup_{i=1}^3 \mathcal{B}_i$;
        \item $\mathcal{B}_5=\text{the set of all birth blocks}$. 
    \end{itemize}

    We first prove that in each individual type, the decomposition is a direct sum decomposition. 
    The proof process for the first four types is easy, but we need to make some small efforts to prove the fifth type. 
    \vspace{0.3cm}
    \begin{proposition}\label{6}
        Let $\mathcal{B}_i$ be a fixed block type.
		The submodules $\mathbb{M}_B$, where $B$ ranges over all blocks of the block type $\mathcal{B}_i$, are in direct sum, that is $\underset{B\in \mathcal{B}_i}{\sum} \mathbb{M}_B\cong \underset{B\in \mathcal{B}_i}{\bigoplus} \mathbb{M}_B$. 
    \end{proposition}
    \begin{proof}
		Let $\boldsymbol{t} \in \mathbb{R}^3$. 
		We only need to prove the equation, $\underset{B\in \mathcal{B}_i}{\sum} (\mathbb{M}_B)_{\boldsymbol{t}}\cong \underset{B\in \mathcal{B}_i}{\bigoplus} (\mathbb{M}_B)_{\boldsymbol{t}}$
		
		\textbf{Case $\mathcal{B}_i$ with $i=1,2,3$}: We only need to prove the case of $i=1$, and the proof for $i=2, 3$ is similar. 
		From Lemma\ref{4}, we can know that $\{(\text{Im}_{c_1,\boldsymbol{t}}^{-},\text{Im}_{c_1,\boldsymbol{t}}^{+})\}_{{c_1}^+ \ni t_1}$ is disjoint. 
        Taking the intersection of all the spaces in this collection with $\text{Im}_{c_2,\boldsymbol{t}}^{+}\cap \text{Im}_{c_3,\boldsymbol{t}}^{+}$, we deduce that 
        $$\{(\text{Im}_{c_1,\boldsymbol{t}}^{-}\cap \text{Im}_{c_2,\boldsymbol{t}}^{+} \cap \text{Im}_{c_3,\boldsymbol{t}}^{+},\text{Im}_{c_1,\boldsymbol{t}}^{+}\cap \text{Im}_{c_2,\boldsymbol{t}}^{+} \cap \text{Im}_{c_3,\boldsymbol{t}}^{+})\}_{{c_1}^+ \ni t_1} = \{(\text{Im}_{B,\boldsymbol{t}}^{-},\text{Im}_{B,\boldsymbol{t}}^{+})\}_{B:\mathcal{B}_1 \ni \boldsymbol{t}}$$ 
		is also disjoint. 
        Hence, by Lemma\ref{9}, the collection of subspaces $\{(\mathbb{M}_B)_{\boldsymbol{t}}\}_{B:\mathcal{B}_1 \ni \boldsymbol{t}}$ is in direct sum.

		\textbf{Case $\mathcal{B}_4$}: Consider any finite collection of distinct death quadrants $B_{1},B_{2},\cdots,B_{m}$ that contain $\boldsymbol{t}$. 
        Since all of them are distinct, there must exist one (denoted as $B_{1}$) that is not contained within the union of the others. 
        Therefore, there exists some $\boldsymbol{u}\geq \boldsymbol{t}$ such that $\boldsymbol{u}\in B_{1}-\bigcup_{i\geq 2} B_{i}$. 
		Suppose there is some relation $\sum_{i=1}^{m} \alpha_i =0$ with $\alpha_i \in (\mathbb{M}_{B_{i}})_{\boldsymbol{t}}$ non-zero for all $i$. 
        Due to the linearity of $\rho_{\boldsymbol{t}}^{\boldsymbol{u}}$, it follows that $\sum_{i=1}^{m} \rho_{\boldsymbol{t}}^{\boldsymbol{u}}(\alpha_i) = 0$. 
		However, $\rho_{\boldsymbol{t}}^{\boldsymbol{u}}(\alpha_i)=0$ for any $i \geq 2$ and $\rho_{\boldsymbol{t}}^{\boldsymbol{u}}(\alpha_1) \neq 0$ due to $\boldsymbol{u}\in B_{1}-\bigcup_{i\geq 2} B_{i}$. This raises a contradiction. 
		
		\textbf{Case $\mathcal{B}_5$}: It suffices to show that, for any finite collection of different birth quadrants $B_1, B_2,\cdots, B_m$, there exists at least one of them ( e.g., $\mathbb{B}_1$) whose corresponding subspace $(\mathbb{M}_{B_1})_{\boldsymbol{t}}\subseteq M_{\boldsymbol{t}}$ is in direct sum with the subspaces corresponding to the other blocks in the collection. 
        Therefore, the result is obtained by a straightforward induction on the size $m$ of the collection. 
		
		Let $B_1,B_2,\cdots,B_m$ be such a collection and each block $B_i={c_{1,i}}^+ \times {c_{2,i}}^+ \times {c_{3,i}}^+$. By reordering if necessary, we can suppose that $B_1$ satisfies
		$${c_{1,1}}^+ \subseteq \underset{i>1}{\bigcap} {c_{1,i}}^+ , 
		{c_{2,1}}^+ \subseteq \underset{c_{1,i}=c_{1,1}}{\underset{i>1}{\bigcap}} {c_{2,i}}^+ , 
        {c_{3,1}}^+ \subseteq \underset{c_{2,i}=c_{2,1}}{\underset{c_{1,i}=c_{1,1}}{\underset{i>1}{\bigcap}}}{c_{3,i}}^+ .
		$$
		
        From the assumption of $B_1$, we can get that $B_1$ does not contain any other blocks. 
        Therefore, by reordering, we can divide these blocks into two subcollections: the ones (denoted as $B_2,\cdots, B_k$) contain $B_1$, while the others (denoted as $B_{k+1},\cdots, B_m$) neither contain $B_1$ nor are not contained by $B_1$.

        In a manner analogous to the proof of Proposition 6.6 in Cochoy and Oudot's work\cite{cochoy2020decomposition}, we deduce that $(\mathbb{M}_{B_1})_{\boldsymbol{t}} \cap (\sum_{i=2}^{m}(\mathbb{M}_{B_i})_{\boldsymbol{t}})\subseteq F_{B_1,\boldsymbol{t}}^{-}$. 
		Note $(\mathbb{M}_{B_1})_{\boldsymbol{t}} \cap F_{B_1,\boldsymbol{t}}^{-} = 0$, then the result follows. 
    \end{proof}
    \vspace{0.3cm}
    \begin{proposition}
		The submodules $\underset{B:\mathcal{B}_1}{\bigoplus} \mathbb{M}_B$, $\underset{B:\mathcal{B}_2}{\bigoplus} \mathbb{M}_B$, $\underset{B:\mathcal{B}_3}{\bigoplus} \mathbb{M}_B$, $\underset{B:\mathcal{B}_4}{\bigoplus} \mathbb{M}_B$ and $\underset{B:\mathcal{B}_5}{\bigoplus} \mathbb{M}_B$ are in direct sum, that is 
		$$\underset{B:\mathcal{B}_1}{\bigoplus} \mathbb{M}_B + \underset{B:\mathcal{B}_2}{\bigoplus} \mathbb{M}_B + \underset{B:\mathcal{B}_3}{\bigoplus} \mathbb{M}_B + \underset{B:\mathcal{B}_4}{\bigoplus} \mathbb{M}_B + \underset{B:\mathcal{B}_5}{\bigoplus} \mathbb{M}_B  \cong \underset{B:\text{blocks}}{\bigoplus} \mathbb{M}_B $$
    \end{proposition}
    \begin{proof}
		We will divide the proof into four parts, 
		\begin{itemize}
		  \item $(\underset{B:\mathcal{B}_5}{\bigoplus}(\mathbb{M}_B)_{\boldsymbol{t}})\cap (\underset{B:\mathcal{B}_1}{\bigoplus}(\mathbb{M}_B)_{\boldsymbol{t}} + \underset{B:\mathcal{B}_2}{\bigoplus}(\mathbb{M}_B)_{\boldsymbol{t}} + \underset{B:\mathcal{B}_3}{\bigoplus}(\mathbb{M}_B)_{\boldsymbol{t}} + \underset{B:\mathcal{B}_4}{\bigoplus}(\mathbb{M}_B)_t)=0$
		  \item $(\underset{B:\mathcal{B}_1}{\bigoplus}(\mathbb{M}_B)_{\boldsymbol{t}})\cap (\underset{B:\mathcal{B}_2}{\bigoplus}(\mathbb{M}_B)_{\boldsymbol{t}} + \underset{B:\mathcal{B}_3}{\bigoplus}(\mathbb{M}_B)_{\boldsymbol{t}} + \underset{B:\mathcal{B}_4}{\bigoplus}(\mathbb{M}_B)_{\boldsymbol{t}})=0$
		  \item $(\underset{B:\mathcal{B}_2}{\bigoplus}(\mathbb{M}_B)_{\boldsymbol{t}})\cap (\underset{B:\mathcal{B}_3}{\bigoplus}(\mathbb{M}_B)_{\boldsymbol{t}} + \underset{B:\mathcal{B}_4}{\bigoplus}(\mathbb{M}_B)_{\boldsymbol{t}})=0$
		  \item $(\underset{B:\mathcal{B}_3}{\bigoplus}(\mathbb{M}_B)_{\boldsymbol{t}})\cap (\underset{B:\mathcal{B}_4}{\bigoplus}(\mathbb{M}_B)_{\boldsymbol{t}})=0$
		\end{itemize}
		
		(1) prove that $(\underset{B:\mathcal{B}_5}{\bigoplus}(\mathbb{M}_B)_{\boldsymbol{t}})\cap (\underset{B:\mathcal{B}_1}{\bigoplus}(\mathbb{M}_B)_{\boldsymbol{t}} + \underset{B:\mathcal{B}_2}{\bigoplus}(\mathbb{M}_B)_{\boldsymbol{t}} + \underset{B:\mathcal{B}_3}{\bigoplus}(\mathbb{M}_B)_{\boldsymbol{t}} + \underset{B:\mathcal{B}_4}{\bigoplus}(\mathbb{M}_B)_{\boldsymbol{t}})=0$ 
		
		Note that if $\boldsymbol{u}=(u_1,u_2,u_3) \in \mathbb{R}^3$ is large enough, then we can know that $\boldsymbol{u} \in B$ for any block $B \in \mathcal{B}_5$ but $\boldsymbol{u}$ is not in any other blocks. we only need demand $u_1,u_2,u_3$ are large enough. 
		
		Let $\alpha$ be a non-zero vector and be in the intersection. 
        It can be decomposed as a linear combination of non-zero vectors $\alpha_1,\cdots,\alpha_n$ from the summands of a finite number of blocks $B_1, B_2,\cdots, B_n$ in $\mathcal{B}_5$. 
        Simultaneously, $\alpha$ can be decomposed as a linear combination of non-zero vectors $\beta_1,\cdots,\beta_m$ from the summands of a finite number of blocks $B_1^{\prime},\cdots, B_m^{\prime}$ of other types: $\sum_{i=1}^{n}\alpha_i=\alpha=\sum_{j=1}^{m}\beta_j$. 
		
		Select a point $\boldsymbol{u} \in \mathbb{R}^3$ so that $\boldsymbol{u}$ is sufficiently large to lie outside the blocks $B_1^{\prime},\cdots,B_n^{\prime}$. 
        What's more, $\boldsymbol{u}$ still lies in the birth quadrants $B_1,\cdots,B_n$. 
		Thus $\rho_{\boldsymbol{t}}^{\boldsymbol{u}}(\sum_{i=1}^{n}\alpha_i)\neq 0$ but $\rho_{\boldsymbol{t}}^{\boldsymbol{u}}(\sum_{j=1}^{m}\beta_i) = 0$. 
		This is a contradiction. 
		
		(2) prove that $(\underset{B:\mathcal{B}_1}{\bigoplus}(\mathbb{M}_B)_{\boldsymbol{t}})\cap (\underset{B:\mathcal{B}_2}{\bigoplus}(\mathbb{M}_B)_{\boldsymbol{t}} + \underset{B:\mathcal{B}_3}{\bigoplus}(\mathbb{M}_B)_{\boldsymbol{t}} + \underset{B:\mathcal{B}_4}{\bigoplus}(\mathbb{M}_B)_{\boldsymbol{t}})=0$, $(\underset{B:\mathcal{B}_2}{\bigoplus}(M_B)_{\boldsymbol{t}})\cap (\underset{B:\mathcal{B}_3}{\bigoplus}(M_B)_{\boldsymbol{t}} \oplus \underset{B:\mathcal{B}_4}{\bigoplus}(M_B)_{\boldsymbol{t}})=0$ and $(\underset{B:\mathcal{B}_3}{\bigoplus}(\mathbb{M}_B)_{\boldsymbol{t}})\cap (\underset{B:\mathcal{B}_4}{\bigoplus}(\mathbb{M}_B)_{\boldsymbol{t}})=0$. 
		
		Similar to (1), we can also choose a point $\boldsymbol{u} \in \mathbb{R}^3$ so that it lies outside the blocks in $\mathcal{B}_1$ but is not in the blocks in $\mathcal{B}_2,\mathcal{B}_3,\mathcal{B}_4$. we need only to demand $u_1,u_2$ are large enough. The remaining processes are almost identical to (1).
		
		(3) prove that $(\underset{B:\mathcal{B}_2}{\bigoplus}(\mathbb{M}_B)_{\boldsymbol{t}})\cap (\underset{B:\mathcal{B}_3}{\bigoplus}(\mathbb{M}_B)_{\boldsymbol{t}} + \underset{B:\mathcal{B}_4}{\bigoplus}(\mathbb{M}_B)_{\boldsymbol{t}})=0$ and $(\underset{B:\mathcal{B}_3}{\bigoplus}(\mathbb{M}_B)_{\boldsymbol{t}})\cap (\underset{B:\mathcal{B}_4}{\bigoplus}(\mathbb{M}_B)_{\boldsymbol{t}})=0$. 
		
		They are treated similarly to (2). 
    \end{proof}
    \vspace{0.3cm}
    Subsequently, we will prove that $\mathbb{M}=\underset{B: block}{\sum}\mathbb{M}_B$. Then we need the notation of covering of sections\cite{crawley2015decomposition}. 
    For any collection of sections $\{(F_{\lambda}^-,F_{\lambda}^-)\}_{\lambda \in \Lambda}$, we say that $\{(F_{\lambda}^-,F_{\lambda}^-)\}_{\lambda \in \Lambda}$ covers a vector space $U$ if for every proper subspace $X\subsetneq U$ there exists a $\lambda \in \Lambda$ satisfying $$X+F_{\lambda}^{-}\neq X+F_{\lambda}^{+}.$$ 
    This collection is said to strongly cover $U$, if for all subspaces $X \subsetneq U$ and $Z \nsubseteq X$ there exists a $\lambda \in \Lambda$ so that $$ X+(F_{\lambda}^{-}\cap Z) \neq X+(F_{\lambda}^{+}\cap Z).$$
    The validity of employing covering sections is substantiated by the subsequent lemma from the reference \cite{crawley2015decomposition}.
    \vspace{0.3cm}
    \begin{lemma}{\cite{crawley2015decomposition}}
		Let $\{(F_{\lambda}^{-},F_{\lambda}^{+})\}_{\lambda \in \Lambda}$ be a collection of sections that covers $U$. 
        For every $\lambda \in \Lambda$, suppose $V_{\lambda}$ is a subspace of $U$ satisfying $F_{\lambda}^{+}=V_{\lambda} \oplus F_{\lambda}^{-}$. 
        It follows that, $U=\sum_{\lambda \in \Lambda} V_{\lambda}$. 
    \end{lemma}
    \vspace{0.3cm}
    \begin{lemma}{\cite{crawley2015decomposition}}
        Let $\{(F_{\lambda}^{-},F_{\lambda}^{+})\}_{\lambda \in \Lambda}$ and $\{G_{\sigma}^{-},G_{\sigma}^{+}\}_{\sigma \in \Sigma}$ be two collections of sections, where the former covers $U$ and the latter strongly covers $U$. 
        Then the following collection covers $U$: 
		$$\{(F_{\lambda}^{-}+G_{\sigma}^{-}\cap F_{\lambda}^{+},F_{\lambda}^{-}+G_{\sigma}^{+}\cap F_{\lambda}^{+})\}_{(\lambda \times \sigma)\in \Lambda \times \Sigma}.$$
    \end{lemma}
    \vspace{0.3cm}
    \begin{lemma}{\cite{crawley2015decomposition}}\label{5}
		Given a fixed $\boldsymbol{t}=(t_1,t_2,t_3)\in \mathbb{R}^3$, for any subsets $X\subsetneq \mathbb{M}_{\boldsymbol{t}}$ and $Z\nsubseteq X$, there is a cut $c_1$ with $t_1 \in {c_1}^+$ such that $\text{Im}_{c_1,\boldsymbol{t}}^{-}\cap Z \subseteq X \nsupseteq \text{Im}_{c_1,\boldsymbol{t}}^{+} \cap Z$. Similarly, there are cuts $c_2$ with $t_2 \in {c_2}^+$ and $c_3$ with $t_3 \in {c_3}^+$ such that $\text{Im}_{c_2,\boldsymbol{t}}^{-}\cap Z \subseteq X \nsupseteq \text{Im}_{c_2,\boldsymbol{t}}^{+} \cap Z$ and $\text{Im}_{c_3,\boldsymbol{t}}^{-}\cap Z \subseteq X \nsupseteq \text{Im}_{c_3,\boldsymbol{t}}^{+} \cap Z$. Same for kernels. 
    \end{lemma}
    \vspace{0.3cm}
    Next, we will prove that the 3-parameter persistence module $\mathbb{M}$, which is \textbf{pfd} and satisfies 3-parameter strong exactness, is the direct sum of block modules.

    Before proving the main theorem, we need to redivide blocks.
    \begin{itemize}
		\item $\mathcal{B}_1=\{ B| {c_2}^- = {c^2}^+ = {c_3}^- = {c^3}^+ =\emptyset\text{ and } {c_1}^-\neq \emptyset\neq {c^1}^+ \}$; 
        \item $\mathcal{B}_2=\{ B| {c_1}^- = {c^1}^+ = {c_3}^- = {c^3}^+ =\emptyset\text{ and } {c_2}^-\neq \emptyset\neq {c^2}^+ \}$; 
        \item $\mathcal{B}_3=\{ B| {c_1}^- = {c^1}^+ = {c_2}^- = {c^2}^+ =\emptyset\text{ and } {c_3}^-\neq \emptyset\neq {c^3}^+ \}$; 
		\item $\mathcal{B}_{4}=\text{the set of all death blocks}\setminus\{\mathbb{R}^n \} $
		\item $\mathcal{B}_{5}=\text{the set of all birth blocks}$
    \end{itemize}

    To prove that the direct sum decomposition of $\mathbb{M}$, we need to define a new 3-parameter persistence module $\tilde{\mathbb{M}}:(\mathbb{R}^3,\leq) \to \textbf{Vec}_{\Bbbk}$, which is a submodule of $\mathbb{M}$, defined as $\tilde{\mathbb{M}}_{\boldsymbol{t}}:=F_{\mathbb{R}^3,\boldsymbol{t}}^{-}=V_{\mathbb{R}^3,\boldsymbol{t}}^{-}=\text{Im}_{\mathbb{R}^3,\boldsymbol{t}}^{+}\cap \text{Ker}_{\mathbb{R}^3,\boldsymbol{t}}^{-}$. 
    Let $X=\tilde{\mathbb{M}}_{\boldsymbol{t}}+\underset{B:\text{birth and layer}}{\sum}(\mathbb{M}_B)_{\boldsymbol{t}}$. 
    Based on the definition of $\tilde{\mathbb{M}}$, it is natural to conjecture that the submodule $\tilde{\mathbb{M}}$ is spanned by the block modules corresponding to with death blocks that are proper subsets of $\mathbb{R}^3$. 
    \vspace{0.3cm}
    \begin{proposition}
		$\mathbb{M}=\tilde{\mathbb{M}}+\underset{B:\mathcal{B}_1\cup\mathcal{B}_2\cup\mathcal{B}_3\cup\mathcal{B}_5}{\bigoplus}\mathbb{M}_B$.
    \end{proposition}
    \begin{proof}
		Given a fixed $\boldsymbol{t}\in \mathbb{R}^3$, let $X=\tilde{\mathbb{M}}_{\boldsymbol{t}}+\underset{B:\mathcal{B}_1\cup\mathcal{B}_2\cup\mathcal{B}_3\cup\mathcal{B}_5}{\bigoplus}(\mathbb{M}_B)_{\boldsymbol{t}}$. 
		Suppose for a contraction that $X\subsetneq \mathbb{M}_{\boldsymbol{t}}$. 
        Then apply Lemma\ref{5} with $Z=\mathbb{M}_{\boldsymbol{t}}$ to get a cut $c_1$ such that $t_1 \in {c_1}^+$ and $\text{Im}_{c_1,\boldsymbol{t}}^{-} \subseteq X \nsupseteq \text{Im}_{c_1,\boldsymbol{t}}^{+}$. 
		Again, use Lemma\ref{5} with $Z=\text{Im}_{c_1,\boldsymbol{t}}^{+}$ to get a cut $c_2$ such that $t_2 \in {c_2}^+$ and $\text{Im}_{c_1,\boldsymbol{t}}^{+}\cap \text{Im}_{c_2,\boldsymbol{t}}^{-} \subseteq X \nsupseteq \text{Im}_{c_1,\boldsymbol{t}}^{+}\cap \text{Im}_{c_2,\boldsymbol{t}}^{+}$. 
		Again, use Lemma\ref{5} with $Z=\text{Im}_{c_1,\boldsymbol{t}}^{+}\cap \text{Im}_{c_2,\boldsymbol{t}}^{+}$ to find a cut $c_3$ so that $t_3 \in {c_3}^+$ and $\text{Im}_{c_1,\boldsymbol{t}}^{+}\cap \text{Im}_{c_2,\boldsymbol{t}}^{+} \cap \text{Im}_{c_3,\boldsymbol{t}}^{-}\subseteq X \nsupseteq \text{Im}_{c_1,\boldsymbol{t}}^{+}\cap \text{Im}_{c_2,\boldsymbol{t}}^{+}\cap \text{Im}_{c_3,\boldsymbol{t}}^{+}$. 
		
		If ${c_1}^-={c_2}^-={c_3}^-=\emptyset$, then 
		$$\text{Im}_{c_1,\boldsymbol{t}}^{+}\cap \text{Im}_{c_2,\boldsymbol{t}}^{+}\cap \text{Im}_{c_3,\boldsymbol{t}}^{+} = \text{Im}_{\mathbb{R}^3,\boldsymbol{t}}^{+} = F_{\mathbb{R}^3,\boldsymbol{t}}^{+} = F_{\mathbb{R}^3,\boldsymbol{t}}^{-} + (\mathbb{M}_{\mathbb{R}^3})_{\boldsymbol{t}} =\tilde{\mathbb{M}}_{\boldsymbol{t}} + (\mathbb{M}_{\mathbb{R}^3})_{\boldsymbol{t}} \subseteq X .$$ 
		However, our selection of $c_1,c_2,c_3$ ensures that $\text{Im}_{c_1,\boldsymbol{t}}^{+}\cap \text{Im}_{c_2,\boldsymbol{t}}^{+}\cap \text{Im}_{c_3,\boldsymbol{t}}^{+} \nsubseteq  X$. This is a contradiction. 
		Thus ${c_1}^{-} \neq \emptyset$ or ${c_2}^{-} \neq \emptyset$ or ${c_3}^{-} \neq \emptyset$.
        
		We distinguish these cases below: 
        These cases are divided as follows:
		
		\textbf{Case} ${c_1}^{-} \neq \emptyset, {c_2}^{-} \neq \emptyset, {c_3}^{-} \neq \emptyset$.
		Let the block $B={c_1}^+ \times {c_2}^+ \times {c_3}^+$. 
		We have $\text{Im}_{B,\boldsymbol{t}}^{+}=\text{Im}_{c_1,\boldsymbol{t}}^{+} \cap \text{Im}_{c_2,\boldsymbol{t}}^{+} \cap \text{Im}_{c_3,\boldsymbol{t}}^{+} = F_{B,\boldsymbol{t}}^{+} \nsubseteq X$. But 
		\begin{equation*}
		  \begin{aligned}
				F_{B,\boldsymbol{t}}^{-} &= \text{Im}_{B,\boldsymbol{t}}^{-}+\text{Ker}_{B,\boldsymbol{t}}^{-}\cap \text{Im}_{B,\boldsymbol{t}}^{+}\\
				&\subseteq \text{Im}_{B,\boldsymbol{t}}^{-}+ (\text{Im}_{c_2,\boldsymbol{t}}^{-}\cap \text{Im}_{c_3,\boldsymbol{t}}^{-}+\text{Im}_{c_1,\boldsymbol{t}}^{-}\cap \text{Im}_{c_3,\boldsymbol{t}}^{-}+\text{Im}_{c_1,\boldsymbol{t}}^{-}\cap \text{Im}_{c_2,\boldsymbol{t}}^{-})\cap \text{Im}_{B,\boldsymbol{t}}^{+}\\
				&\subseteq \text{Im}_{B,\boldsymbol{t}}^{-} \subseteq \text{Im}_{c_1,\boldsymbol{t}}^{-}+\text{Im}_{c_1,\boldsymbol{t}}^{+}\cap \text{Im}_{c_2,\boldsymbol{t}}^{-}+\text{Im}_{c_1,\boldsymbol{t}}^{+}\cap \text{Im}_{c_2,\boldsymbol{t}}^{+}\cap \text{Im}_{c_3,\boldsymbol{t}}^{-} \subseteq X
		  \end{aligned}
		\end{equation*}
		by Lemma\ref{2}. 
		Note that $F_{B,\boldsymbol{t}}^{+}=F_{B,\boldsymbol{t}}^{-}\oplus (\mathbb{M}_B)_{\boldsymbol{t}}$. 
		Then we get a contradiction, $(\mathbb{M}_B)_{\boldsymbol{t}} \nsubseteq X$. 
		
		\textbf{Case} ${c_1}^{-} \neq \emptyset, {c_2}^{-} \neq \emptyset, {c_3}^{-} = \emptyset$. 
		Let the block $B={c_1}^+ \times {c_2}^+ \times {c_3}^+$. 
		We have $\text{Im}_{B,\boldsymbol{t}}^{+}=\text{Im}_{c_1,\boldsymbol{t}}^{+} \cap \text{Im}_{c_2,\boldsymbol{t}}^{+} \cap \text{Im}_{c_3,\boldsymbol{t}}^{+} = F_{B,\boldsymbol{t}}^{+} \nsubseteq X$, but $F_{B,\boldsymbol{t}}^{-} = \text{Im}_{B,\boldsymbol{t}}^{-}+\text{Ker}_{B,\boldsymbol{t}}^{-}\cap \text{Im}_{B,\boldsymbol{t}}^{+} \subseteq \text{Im}_{B,\boldsymbol{t}}^{-}+ (\text{Im}_{c_2,\boldsymbol{t}}^{-}\cap \text{Im}_{c_3,\boldsymbol{t}}^{+}+\text{Im}_{c_1,\boldsymbol{t}}^{-}\cap \text{Im}_{c_3,\boldsymbol{t}}^{+}+\text{Im}_{c_1,\boldsymbol{t}}^{-}\cap \text{Im}_{c_2,\boldsymbol{t}}^{-})\cap \text{Im}_{B,\boldsymbol{t}}^{+} \subseteq \text{Im}_{B,\boldsymbol{t}}^{-} \subseteq X$ by Lemma\ref{2}. 
		Note that $F_{B,\boldsymbol{t}}^{+}=F_{B,\boldsymbol{t}}^{-}\oplus (\mathbb{M}_B)_{\boldsymbol{t}}$. 
        Thus, $(\mathbb{M}_B)_{\boldsymbol{t}} \nsubseteq X$. This is a contradiction.
		
		Similarly, we can prove these cases that ${c_1}^{-} \neq \emptyset, {c_2}^{-} = \emptyset, {c_3}^{-} \neq \emptyset$ and ${c_1}^{-} = \emptyset, {c_2}^{-} \neq \emptyset, {c_3}^{-} \neq \emptyset$. 
		
		\textbf{Case} ${c_1}^{-} \neq \emptyset, {c_2}^{-} = \emptyset, {c_3}^{-} = \emptyset$. 
		By Lemma\ref{5}, applied with $Z=\text{Im}_{c_1,\boldsymbol{t}}^{+}\cap \text{Im}_{c_2,\boldsymbol{t}}^{+}\cap \text{Im}_{c_3,\boldsymbol{t}}^{+}$, there is a cut $c^1$ such that $\boldsymbol{t}\in {c^1}^-$ and 
		$$\text{Im}_{c_1,\boldsymbol{t}}^{+}\cap \text{Im}_{c_2,\boldsymbol{t}}^{+}\cap \text{Im}_{c_3,\boldsymbol{t}}^{+}\cap \text{Ker}_{c^1,\boldsymbol{t}}^{-} \subseteq X \nsupseteq \text{Im}_{c_1,\boldsymbol{t}}^{+}\cap \text{Im}_{c_2,\boldsymbol{t}}^{+}\cap \text{Im}_{c_3,\boldsymbol{t}}^{+}\cap \text{Ker}_{c^1,\boldsymbol{t}}^{+}.$$
		Let the block $B=({c_1}^+ \cap {c^1}^-) \times {c_2} \times {c_3}$. 
		Using Lemma\ref{2}, we have $\text{Ker}_{c^1,\boldsymbol{t}}^{+}\subseteq \text{Im}_{c_2,\boldsymbol{t}}^{+} \cap \text{Im}_{c_3,\boldsymbol{t}}^{+}$, $\text{Ker}_{c^2,\boldsymbol{t}}^{-}\subseteq \text{Im}_{c_1,\boldsymbol{t}}^{-} \cap \text{Im}_{c_3,\boldsymbol{t}}^{+}$ and $\text{Ker}_{c^3,\boldsymbol{t}}^{-}\subseteq \text{Im}_{c_1,\boldsymbol{t}}^{-} \cap \text{Im}_{c_2,\boldsymbol{t}}^{+}$. Then 
		$$\text{Im}_{B,\boldsymbol{t}}^{-}=\text{Im}_{c_1,\boldsymbol{t}}^{-}\cap \text{Im}_{c_2,\boldsymbol{t}}^{+} \cap \text{Im}_{c_3,\boldsymbol{t}}^{+} \subseteq X$$
		$$\text{Ker}_{B,\boldsymbol{t}}^{+}\cap \text{Im}_{B,\boldsymbol{t}}^{+}\supseteq \text{Im}_{c_1,\boldsymbol{t}}^{+}\cap \text{Im}_{c_2,\boldsymbol{t}}^{+}\cap \text{Im}_{c_3,\boldsymbol{t}}^{+}\cap \text{Ker}_{c^1,\boldsymbol{t}}^{+} \nsubseteq X$$
		\begin{equation*}
		  \begin{aligned}
				&\text{Im}_{B,\boldsymbol{t}}^{+}\cap \text{Ker}_{B,\boldsymbol{t}}^{-}=\text{Im}_{c_1,\boldsymbol{t}}^{+} \cap \text{Im}_{c_2,\boldsymbol{t}}^{+} \cap \text{Im}_{c_3,\boldsymbol{t}}^{+} \cap (\text{Ker}_{c^1,\boldsymbol{t}}^{-}+\text{Ker}_{c^2,\boldsymbol{t}}^{-}+\text{Ker}_{c^3,\boldsymbol{t}}^{-})\\
				&\subseteq \text{Im}_{c_1,\boldsymbol{t}}^{+} \cap \text{Im}_{c_2,\boldsymbol{t}}^{+} \cap \text{Im}_{c_3,\boldsymbol{t}}^{+} \cap (\text{Ker}_{c^1,\boldsymbol{t}}^{-}+\text{Im}_{c_1,\boldsymbol{t}}^{-} \cap \text{Im}_{c_3,\boldsymbol{t}}^{+}+\text{Im}_{c_1,\boldsymbol{t}}^{-} \cap \text{Im}_{c_2,\boldsymbol{t}}^{+})\\
				&=\text{Im}_{c_1,\boldsymbol{t}}^{+} \cap \text{Im}_{c_2,\boldsymbol{t}}^{+} \cap \text{Im}_{c_3,\boldsymbol{t}}^{+} \cap \text{Ker}_{c^1,\boldsymbol{t}}^{-} + \text{Im}_{c_1,\boldsymbol{t}}^{-} \cap \text{Im}_{c_2,\boldsymbol{t}}^{+} \cap \text{Im}_{c_3,\boldsymbol{t}}^{+} \subseteq X
		  \end{aligned}
		\end{equation*}
		Thus, $F_{B,\boldsymbol{t}}^{-} \subseteq X \nsupseteq F_{B,\boldsymbol{t}}^{+}$. Hence, $(\mathbb{M}_B)_{\boldsymbol{t}} \nsubseteq X$. This is a contradiction.
		
		Similarly, we can prove these cases that ${c_1}^{-} = \emptyset, {c_2}^{-} \neq \emptyset, {c_3}^{-}=\emptyset$ and ${c_1}^{-} = \emptyset, {c_2}^{-} = \emptyset, {c_3}^{-} \neq \emptyset$. 
    \end{proof}
    \vspace{0.3cm}
    \begin{lemma}
		$(\tilde{\mathbb{M}}+\underset{B:\mathcal{B}_1 \cup \mathcal{B}_2 \cup \mathcal{B}_3}{\bigoplus} \mathbb{M}_B)+\underset{B:\mathcal{B}_5}{\bigoplus} \mathbb{M}_B = (\tilde{\mathbb{M}}+\underset{B:\mathcal{B}_1\cup \mathcal{B}_2 \cup \mathcal{B}_3}{\bigoplus} \mathbb{M}_B)\oplus \underset{B:\mathcal{B}_5}{\bigoplus} \mathbb{M}_B$
    \end{lemma}
    \begin{proof}
        Assume the opposite, and let $\boldsymbol{t}\in \mathbb{R}^3$ so that $(\tilde{\mathbb{M}}+\underset{B:\mathcal{B}_1 \cup \mathcal{B}_2 \cup \mathcal{B}_3}{\bigoplus} \mathbb{M}_B)_{\boldsymbol{t}} \cap (\underset{B:\mathcal{B}_5}{\bigoplus} \mathbb{M}_B)_{\boldsymbol{t}} \neq \emptyset$. 
		Then there exist $\alpha \in \tilde{\mathbb{M}}_{\boldsymbol{t}}$, $\alpha_1 \in (\mathbb{M}_{B_1})_{\boldsymbol{t}}, \cdots, \alpha_r \in (\mathbb{M}_{B_r})_{\boldsymbol{t}}$ and $\alpha_{r+1} \in (\mathbb{M}_{B_{r+1}})_{\boldsymbol{t}}, \cdots, \alpha_n \in (\mathbb{M}_{B_n})_{\boldsymbol{t}}$, such that $B_1,\cdots,B_r$ are in $\mathcal{B}_1 \cup \mathcal{B}_2 \cup \mathcal{B}_3$, $B_{r+1},\cdots,B_{n}$ are in $\mathcal{B}_5$, and we have 
		$$\alpha+\sum_{i=1}^{r}\alpha_i=\sum_{j=r+1}^{n}\alpha_j \neq 0 .$$ 
		
		Because of the shape of these blocks in $\mathcal{B}_1 \cup \mathcal{B}_2 \cup \mathcal{B}_3$, we may find out some $\boldsymbol{u}\geq \boldsymbol{t}$ such that $\boldsymbol{u}\notin \bigcup_{i=1}^{r}B_i$. 
		What's more, since $\alpha \in \tilde{\mathbb{M}}_{\boldsymbol{t}} = \text{Im}_{\mathbb{R}^3,\boldsymbol{t}}^{+}\cap \text{Ker}_{\mathbb{R}^3,\boldsymbol{t}}^{-} \subseteq \text{Ker}_{\mathbb{R}^3,\boldsymbol{t}}^{-} = \text{Ker}_{c^1,\boldsymbol{t}}^{-}+ \text{Ker}_{c^2,\boldsymbol{t}}^{-}+ \text{Ker}_{c^3,\boldsymbol{t}}^{-}$, we have $\alpha=\alpha_1^{\prime}+\alpha_2^{\prime}+\alpha_3^{\prime}$ for some $\alpha_1^{\prime} \in \text{Ker}_{c_1,\boldsymbol{t}}^{-}$, $\alpha_2^{\prime} \in \text{Ker}_{c_2,\boldsymbol{t}}^{-}$ and $\alpha_3^{\prime} \in \text{Ker}_{c_3,\boldsymbol{t}}^{-}$. 
		By Lemma\ref{1}, there are finite coordinates $x\geq t_1$, $y\geq t_2$ and $z\geq t_3$ such that $\alpha_1^{\prime} \in \text{Ker } \rho_{\boldsymbol{t}}^{(x,t_2,t_3)}$, $\alpha_2^{\prime} \in \text{Ker } \rho_{\boldsymbol{t}}^{(t_1,y,t_3)}$ and $\alpha_3^{\prime} \in \text{Ker } \rho_{\boldsymbol{t}}^{(t_1,t_2,z)}$. 
		Let $\boldsymbol{v}$ be a point with coordinates $(\text{max}\{u_1,x\},\text{max}\{u_2,y\},\text{max}\{u_3,z\})$. 
		Then we obtain
		$$\rho_{\boldsymbol{t}}^{\boldsymbol{v}}(\alpha+\sum_{i=1}^{r}\alpha_i) = 0.$$
		However, because $\rho_{\boldsymbol{t}}^{\boldsymbol{v}}$ restricted to $\bigoplus_{r+1}^{n}(\mathbb{M}_{B_i})_r$ is injective, we have $\rho_{\boldsymbol{t}}^{\boldsymbol{v}}(\sum_{i=r+1}^{n})\alpha_i \neq 0$. 
		This is a contradiction. 
        Thus $(\tilde{\mathbb{M}}+\underset{B:\mathcal{B}_1 \cup \mathcal{B}_2 \cup \mathcal{B}_3}{\bigoplus} \mathbb{M}_B)\cap \underset{B:\mathcal{B}_5}{\bigoplus} \mathbb{M}_B=0$.
    \end{proof}
    \vspace{0.3cm}
    \begin{lemma}
		$\tilde{\mathbb{M}}+\underset{B:\mathcal{B}_1 \cup \mathcal{B}_2 \cup \mathcal{B}_3}{\bigoplus} \mathbb{M}_B = \tilde{\mathbb{M}}\oplus \underset{B:\mathcal{B}_1 \cup \mathcal{B}_2 \cup \mathcal{B}_3}{\bigoplus} \mathbb{M}_B$
    \end{lemma}
    \begin{proof}
		Assume the opposite, and let $\boldsymbol{t}\in \mathbb{R}^3$ so that $(\tilde{\mathbb{M}})_{\boldsymbol{t}}\cap (\underset{B:\mathcal{B}_1 \cup \mathcal{B}_2 \cup \mathcal{B}_3}{\bigoplus} \mathbb{M}_B)_{\boldsymbol{t}} \neq \emptyset$. 
        Then there exist $\alpha \in \tilde{\mathbb{M}}_{\boldsymbol{t}}$, $\alpha_i \in (\mathbb{M}_{B_i})_{\boldsymbol{t}}$ with $i=1,2,\cdots,n$ such that $B_1,B_2,\cdots,B_n$ are in $\mathcal{B}_1 \cup \mathcal{B}_2 \cup \mathcal{B}_3$ and 
		$$\alpha = \sum_{i}^{n} \alpha_i \neq 0.$$
		Assume that $B_1,\cdots,B_k$ are in $\mathcal{B}_1$, $B_{k+1},\cdots,B_{r}$ are in $\mathcal{B}_2$ and $B_{r+1},\cdots,B_n$ are in $\mathcal{B}_3$. And assume that none of the $\alpha_i$'s are zero. 
		Because of the shape of these blocks, we may find a point $\boldsymbol{u}=(u_1,u_2,u_3)=(x,y,t_3)\in \mathbb{R}^3$ such that $\boldsymbol{u} \notin \bigcup_{i=k+1}^{n} B_i$, then $\rho_{\boldsymbol{t}}^{\boldsymbol{u}}(\sum_{i=k+1}^{n} \alpha_i) = 0$. 
		Since the restriction of $\rho_{\boldsymbol{t}}^{\boldsymbol{u}}$ to $\bigoplus_{i=1}^{k}(\mathbb{M}_{B_i})_{\boldsymbol{t}}$ is injective, $\rho_{t}^{u}(\sum_{i=1}^{k} \alpha_i)\neq 0$. 
		
		Let $\beta=\rho_{\boldsymbol{t}}^{\boldsymbol{u}}(\alpha)\in \tilde{\mathbb{M}}_{\boldsymbol{u}}$ and $\beta_i=\rho_{\boldsymbol{t}}^{\boldsymbol{u}}(\alpha_i)\in (\mathbb{M}_{B_i})_{\boldsymbol{u}}$ for $i=1,\cdots,k$
		$$\beta=\sum_{i=1}^{k} \beta_i \neq 0 .$$
		Now, we have $\tilde{\mathbb{M}}_{\boldsymbol{u}} \subseteq \text{Im}_{\mathbb{R}^3,\boldsymbol{u}}^{+} = F_{\mathbb{R}^3,\boldsymbol{u}}^{+}$. 
		From the proof of Lemma\ref{6}, the collection of sections $\{(F_{B_1,\boldsymbol{u}}^{-},F_{B_1,\boldsymbol{u}}^{+}),\cdots,(F_{B_k,\boldsymbol{u}}^{-},F_{B_k,\boldsymbol{u}}^{+})\}$ is disjoint. 
		Note that $F_{\mathbb{R}^3,\boldsymbol{u}}^{+} \subset F_{B_i,\boldsymbol{u}}^{-}$ for every $i$, then the collection of sections $\{(0,F_{\mathbb{R}^3,\boldsymbol{u}}^{+}), (F_{B_1,\boldsymbol{u}}^{-},F_{B_1,\boldsymbol{u}}^{+}),\cdots, (F_{B_k,\boldsymbol{u}}^{-},F_{B_k,\boldsymbol{u}}^{+})\}$ is disjoint.
        Then according to Lemma\ref{10}, $F_{\mathbb{R}^3,\boldsymbol{u}}^{+}$ is in direct sum with $\bigoplus_{i=1}^{k}(\mathbb{M}_{B_{i}})_{\boldsymbol{u}}$. 
        Because $\tilde{\mathbb{M}}_u$ is a subspace of $F_{\mathbb{R}^3,\boldsymbol{u}}^{+}$ and $F_{\mathbb{R}^3,\boldsymbol{u}}^{+}$ is in direct sum with $\bigoplus_{i=1}^{k}(\mathbb{M}_{B_{i}})_{\boldsymbol{u}}$, the result contradicts $\beta=\sum_{i=1}^{r} \beta_i \neq 0$. 
    \end{proof}
    \vspace{0.3cm}
    \begin{corollary}
		$\mathbb{M}=\tilde{\mathbb{M}}\oplus \underset{B:\mathcal{B}_1\cup \mathcal{B}_2 \cup \mathcal{B}_3 \cup \mathcal{B}_5}{\bigoplus}\mathbb{M}_B$
    \end{corollary}
    \vspace{0.3cm}
    Through the above discussion, we have extracted all the block submodules of $\mathbb{M}$ except for dead blocks and prove that they are in direct sum. 
    To prove the main theorem, we only need to prove that submodule $\tilde{\mathbb{M}}$ can also be decomposed as the direct sum of block modules. 
	
    However, we do not directly decompose $\tilde{\mathbb{M}}$ but rather decompose the duality $\tilde{\mathbb{M}}^*$. Let $\tilde{\mathbb{M}}^*$ be the pointwise dual of $\tilde{\mathbb{M}}$, that is $(\tilde{\mathbb{M}}^*)_{\boldsymbol{t}}=Hom(\tilde{\mathbb{M}}_{\boldsymbol{t}},\Bbbk)$. 
    Since the duality is a contravariant functor, $\tilde{\mathbb{M}}^* : ((\mathbb{R}^{op})^3,\geq) \to \textbf{Vec}_{\Bbbk} $ is a persistence module, where $\mathbb{R}^{op}$ denoted the poset $\mathbb{R}$ with the opposite order $\geq$. 
	
    So we need the following result. 
    \vspace{0.3cm}
    \begin{lemma}
		$\tilde{\mathbb{M}}^*$ is \textbf{pfd} and satisfies the 3-parameter strong exactness. 
    \end{lemma}
    \begin{proof}
		Our proof is mainly divided into two parts. The first part is to prove that for any $r\in \mathbb{R}$, $\tilde{\mathbb{M}}_{\{r\}\times \mathbb{R} \times \mathbb{R}}$ satisfies 2-parameter strong exactness, and the proof method for $\tilde{\mathbb{M}}|_{\mathbb{R}\times \{r\}\times \mathbb{R}}$ and $\tilde{\mathbb{M}}|_{\mathbb{R}\times \mathbb{R} \times \{r\}}$ are similar. 
		The second part is to prove that for any $(s_1,s_2,s_3)\leq (t_1,t_2,t_3) \in \mathbb{R}^3$, the morphism $\varphi$ and $\psi$ associated with the persistence module $\tilde{\mathbb{M}}$ are injective and surjective respectively. 
		
		Obviously, $\tilde{\mathbb{M}}$ is \textbf{pfd}, then $\tilde{\mathbb{M}}^*$ is \textbf{pfd}. 
		$$
		\xymatrix{
			N_{(r,s_2,t_3)} \ar[r] & N_{(r,t_2,t_3)}\\
			N_{(r,s_2,s_3)} \ar[r] \ar[u] & N_{(r,t_2,s_3)} \ar[u]\\
		}
		$$
		
		Firstly, let $(r,s_2,s_3) \leq (r,t_2,t_3) \in \mathbb{R}^3$ and take an element $\delta \in \tilde{\mathbb{M}}_{(r,t_2,t_3)}$ that has preimages $\beta \in \tilde{\mathbb{M}}_{(r,t_2,s_3)}$ and $\gamma \in \tilde{\mathbb{M}}_{(r,s_2,t_3)}$. 
		Then, by the 3-parameter strong exactness of $\mathbb{M}$, $\beta$ and $\gamma$ have a shared preimage $\alpha \in \mathbb{M}_{(r,s_2,s_3)}$. 
		Indeed, we can prove that $\alpha \in \tilde{\mathbb{M}}_{(r,s_2,s_3)}$. 
		Obviously, we know that $\alpha \in (\rho_{(r,s_2,s_3)}^{(r,t_2,t_3)})^{-1} \tilde{\mathbb{M}}_{(r,t_2,t_3)} \subseteq (\rho_{(r,s_2,s_3)}^{(r,t_2,t_3)})^{-1}(\text{Ker}_{\mathbb{R}^3,(r,t_2,t_3)}^{-})=\text{Ker}_{\mathbb{R}^3,(r,s_2,s_3)}^{-}$. 
		What's more, because of $\beta \in \tilde{\mathbb{M}}_{(r,t_2,s_3)} \subseteq \text{Im}_{\mathbb{R}^3,(r,t_2,s_3)}^{+}$, for any $\boldsymbol{u} < (r,t_2,s_3) \in \mathbb{R}^3$ with $u_1=r$ and $u_2 = t_2$ there is some preimage $\beta_{\boldsymbol{u}}$ of $\beta$ in $\mathbb{M}_{(r,t_2,u_3)}$ by the 3-parameter strong exactness, implies that there exists a shared preimage $\alpha_{\boldsymbol{u}}$ of $\alpha$ and $\beta_{\boldsymbol{u}}$ in $\mathbb{M}_{\boldsymbol{u}}$. 
		Thus $\alpha \in \text{Im}_{c_3,(r,s_2,s_3)}^{+}$, where $c_3$ is the trivial cut that is ${c_3}^-=\emptyset$. 
		Similarly, we can know that $\alpha \in \text{Im}_{c_1,(r,s_2,s_3)}^{+}$ and $\alpha \in \text{Im}_{c_2,(r,s_2,s_3)}^{+}$, in which ${c_1}^-={c_2}^-=\emptyset$. 
		So $\alpha \in \text{Im}_{\mathbb{R}^3,(r,s_2,s_3)}^{+}$, and therefore $\alpha\in \tilde{\mathbb{M}}_{(r,s_2,s_3)}$. 
		
		In other words, $\tilde{\mathbb{M}}|_{\{r\}\times \mathbb{R}\times \mathbb{R}}$ satisfies the 2-parameter strong exactness. 
		Thus, $\tilde{\mathbb{M}}^*|_{\{r\}\times \mathbb{R}^{op} \times \mathbb{R}^{op}}$ satisfies the 2-parameter strong exactness\cite{cochoy2020decomposition}. 
		
		Secondly, for any $(s_1,s_2,s_3) \leq (t_1,t_2,t_3) \in \mathbb{R}^3$, we get a commutative diagram
		$$
		\xymatrix{
			&\tilde{\mathbb{M}}_{(s_1,t_2,t_3)}\ar[rr]&&\tilde{\mathbb{M}}_{(t_1,t_2,t_3)}\\
			\tilde{\mathbb{M}}_{(s_1,s_2,t_3)}\ar[ru]\ar[rr]&&\tilde{\mathbb{M}}_{(t_1,s_2,t_3)}\ar[ru]&\\
			&\tilde{\mathbb{M}}_{(s_1,t_2,s_3)}\ar[uu]|\hole \ar[rr]|\hole &&\tilde{\mathbb{M}}_{(t_1,t_2,s_3)}\ar[uu]\\
			\tilde{\mathbb{M}}_{(s_1,s_2,s_3)}\ar[rr]\ar[uu]\ar[ru]&&\tilde{\mathbb{M}}_{(t_1,s_2,s_3)}\ar[uu]\ar[ru]&
		}
		$$
		We will denote it as $\mathcal{X}:\mathcal{P}(S) \to \textbf{Vec}_{\Bbbk}$ that $S$ is a set with $|S|=3$, and get the morphism $\varphi:\underset{T\in \mathcal{P}_1(S)}{\text{colim}}\mathcal{X}(T) \to \mathcal{X}(S)$ and morphism $\psi: \mathcal{X}(\emptyset) \to \underset{T\in \mathcal{P}_0(S)}{\text{lim}}\mathcal{X}(T)$. 
		
		Note that for any $\boldsymbol{s} \leq \boldsymbol{t} \in \mathbb{R}^3$, we have $\tilde{\mathbb{M}}_{\boldsymbol{t}}=\text{Im}_{\mathbb{R}^3,\boldsymbol{t}}^{+} \cap \text{Ker}_{\mathbb{R}^3,\boldsymbol{t}}^{-}$, $\rho_{\boldsymbol{s}}^{\boldsymbol{t}}(\text{Im}_{\mathbb{R},\boldsymbol{s}}^{+})=\text{Im}_{\mathbb{R},\boldsymbol{t}}^{+}$ and $(\rho_{\boldsymbol{s}}^{\boldsymbol{t}})^{-1}(\text{Ker}_{\mathbb{R}^3,\boldsymbol{t}}^{-})=\text{Ker}_{\mathbb{R}^3,\boldsymbol{s}}^{-}$. 
		Thus for any $\alpha \in \tilde{\mathbb{M}}_{\boldsymbol{t}}$, we always can find out some $\beta \in \tilde{\mathbb{M}}_{\boldsymbol{s}}$ such that $\rho_{\boldsymbol{s}}^{\boldsymbol{t}}(\beta)=\alpha$. 
		Given $\underset{T\in \mathcal{P}_1(S)}{\text{colim}}\mathcal{X}(T)=\tilde{\mathbb{M}}_{(s_1,t_2,t_3)}\oplus \tilde{\mathbb{M}}_{(t_1,s_2,t_3)} \oplus \tilde{\mathbb{M}}_{(t_1,t_2,s_3)} / \sim $. 
		Then for any $[\alpha+\beta+\gamma] \in \underset{T\in \mathcal{P}_1(S)}{\text{colim}}\mathcal{X}(T)$ satisfying $\varphi([\alpha+\beta+\gamma])=0$ in which $\alpha \in \tilde{\mathbb{M}}_{(s_1,t_2,t_3)}$, $\beta \in \tilde{\mathbb{M}}_{(t_1,s_2,t_3)}$ and $\gamma \in \tilde{\mathbb{M}}_{(t_1,t_2,s_3)}$, we may find out some $\tilde{\gamma} \in \tilde{\mathbb{M}}_{(t_1,t_2,s_3)}$ such that $[\alpha+\beta+\gamma]=[\tilde{\gamma}]$. 
		Since $\varphi([\tilde{\gamma}])=0$, then $\rho_{(t_1,t_2,s_3)}^{(t_1,t_2,t_3)}(\tilde{\gamma})=0$. 
		Therefore, we can find out some common preimage of $\tilde{\gamma}$ and $0 \in \tilde{\mathbb{M}}_{(t_1,s_2,t_3)}$, then $[\alpha+\beta+\gamma]=[\tilde{\gamma}]=0 \in \underset{T\in \mathcal{P}_1(S)}{\text{colim}}\mathcal{X}(T).$ 
		Thus, $\varphi$ is injective. 
		Obviously, $\varphi^*$ is surjective. 

        To proving that $\psi$ is surjective, suppose $\alpha_1 \in \tilde{\mathbb{M}}_{(t_1,s_2,s_3)},\alpha_2 \in \tilde{\mathbb{M}}_{(s_1,t_2,s_3)},\alpha_3 \in \tilde{\mathbb{M}}_{(s_1,s_2,t_3)}$. Because of the 3-parameter strong exactness of $\mathbb{M}$, we may find out $\alpha \in \mathbb{M}_{(s_1,s_2,s_3)}$ such that $\alpha$ is the common preimage of $\alpha_1,\alpha_2,\alpha_3$. 
        Given that  
        \begin{equation*}
            \begin{aligned}
                &\alpha_1\in \tilde{\mathbb{M}}_{(t_1,s_2,s_3)}=\text{Im}_{\mathbb{R}^3,(t_1,s_2,s_3)}^{+} \cap \text{Ker}_{\mathbb{R}^3,(t_1,s_2,s_3)}^{-}, \\
                &{\rho_{(s_1,s_2,s_3)}^{(t_1,s_2,s_3)}}^{-1} \text{Ker}_{\mathbb{R}^3,(t_1,s_2,s_3)}^{-} = \text{Ker}_{\mathbb{R}^3,(s_1,s_2,s_3)}^{-}, \\
                &\text{Im}_{\mathbb{R}^3,(t_1,s_2,s_3)}^{+} = \text{Im}_{c_1,(t_1,s_2,s_3)}^{+} \cap \text{Im}_{c_2,(t_1,s_2,s_3)}^{+} \cap \text{Im}_{c_3,(t_1,s_2,s_3)}^{+}. \\
            \end{aligned}
        \end{equation*}
        Because of Lemma\ref{1}, we may prove that ${\rho_{(s_1,s_2,s_3)}^{(t_1,s_2,s_3)}}^{-1} \text{Im}_{c_2,(t_1,s_2,s_3)}^{+} = \text{Im}_{c_2,(s_1,s_2,s_3)}^{+}$. 
        Lemma\ref{1} told us that there exists some $y \leq s_2$ such that
        \begin{equation*}
            \begin{aligned}
                \text{Im}_{c_2,(t_1,s_2,s_3)}^{+}=\rho_{(t_1,y,s_3)}^{(t_1,s_2,s_3)} \mathbb{M}_{(t_1,y,s_3)},\\
                \text{Im}_{c_2,(s_1,s_2,s_3)}^{+}=\rho_{(s_1,y,s_3)}^{(s_1,s_2,s_3)} \mathbb{M}_{(s_1,y,s_3)}.
            \end{aligned}
        \end{equation*}
        For any $\beta \in {\rho_{(s_1,s_2,s_3)}^{(t_1,s_2,s_3)}}^{-1} \text{Im}_{c_2,(t_1,s_2,s_3)}^{+}$, we may find out $\gamma \in \mathbb{M}_{(t_1,y,s_3)}$ such that $\rho_{(t_1,y,s_3)}^{(s_1,s_2,s_3)} (\gamma) = \rho_{(s_1,s_2,s_3)}^{(t_1,s_2,s_3)}(\beta)$. By the 2-parameter strong exactness, there is a preimage $\delta \in \mathbb{M}_{(s_1,y,s_3)}$ of $\beta$ and $\gamma$, then $\beta \in \rho_{(s_1,y,s_3)}^{(s_1,s_2,s_3)} M_{(s_1,y,s_3)}=\text{Im}_{c_2,(s_1,s_2,s_3)}^{+}$. So we have proven ${\rho_{(s_1,s_2,s_3)}^{(t_1,s_2,s_3)}}^{-1} \text{Im}_{c_2,(t_1,s_2,s_3)}^{+} = \text{Im}_{c_2,(s_1,s_2,s_3)}^{+}$. 
        Similarly, we can prove that ${\rho_{(s_1,s_2,s_3)}^{(t_1,s_2,s_3)}}^{-1} \text{Im}_{c_3,(t_1,s_2,s_3)}^{+} = \text{Im}_{c_3,(s_1,s_2,s_3)}^{+}$, then $\alpha \in \text{Im}_{c_2,(s_1,s_2,s_3)}^{+} \cap \text{Im}_{c_3,(s_1,s_2,s_3)}^{+}$. 
        In the same way, by considering $\alpha$ as a preimage of $\alpha_2$ and $\alpha_3$ respectively, we can prove that $\alpha \in \text{Im}_{c_1,(s_1,s_2,s_3)}^{+} \cap \text{Im}_{c_3,(s_1,s_2,s_3)}^{+}$ and $\alpha \in \text{Im}_{c_1,(s_1,s_2,s_3)}^{+} \cap \text{Im}_{c_2,(s_1,s_2,s_3)}^{+}$. 
        We have proven that $\alpha \in \text{Im}_{c_1,(s_1,s_2,s_3)}^{+} \cap \text{Im}_{c_2,(s_1,s_2,s_3)}^{+} \cap \text{Im}_{c_3,(s_1,s_2,s_3)}^{+} = \text{Im}_{\mathbb{R}^3,(s_1,s_2,s_3)}^+$. 
        Thus 
        $\alpha \in \text{Im}_{\mathbb{R}^3,\boldsymbol{s}}^+ \cap \text{Ker}_{\mathbb{R}^3,\boldsymbol{s}}^-=\tilde{\mathbb{M}}_{\boldsymbol{s}}$, and the morphism $\psi$ is surjective. 
        Obviously, the duality of $\psi$, $\psi^*$, is injective.
        
		So $\tilde{\mathbb{M}}^*$ satisfies the 3-parameter strong exactness. 
		
    \end{proof}
    \vspace{0.3cm}
    By the above lemma, we know that $\tilde{\mathbb{M}}^*$ can also be decomposed like the above decomposition of $\mathbb{M}$. 
    \vspace{0.3cm}
    \begin{lemma}
		For any $\boldsymbol{t}\in (\mathbb{R}^{op})^3$, $\text{Im}_{(\mathbb{R}^{op})^3,\boldsymbol{t}}^{+}(\tilde{\mathbb{M}}^*)=0$. 
    \end{lemma}
    \begin{proof}
		Let $X^{\perp}$ denote the annihilator of any subspace $X\subseteq \tilde{\mathbb{M}}_{\boldsymbol{t}}$:
		$$X^{\perp}=\{\phi(\alpha)=0 \text{ for all } \alpha \in X \}.$$ 
		Because the annihilator operation transforms sums into intersections and kernels into images, 
		then
		\begin{equation*}
		  \begin{aligned}
				(\text{Ker}_{\mathbb{R}^3,\boldsymbol{t}}^{-}(\tilde{\mathbb{M}}))^{\perp}&=(\text{Ker}_{c^1,\boldsymbol{t}}^{-}(\tilde{\mathbb{M}})+\text{Ker}_{c^2,\boldsymbol{t}}^{-}(\tilde{\mathbb{M}})+\text{Ker}_{c^3,\boldsymbol{t}}^{-}(\tilde{\mathbb{M}}))^{\perp}\\
				&=\text{Im}_{c^1,\boldsymbol{t}}^{+}(\tilde{\mathbb{M}}^*) \cap \text{Im}_{c^2,\boldsymbol{t}}^{+}(\tilde{\mathbb{M}}^*) \cap \text{Im}_{c^3,\boldsymbol{t}}^{+}(\tilde{\mathbb{M}}^*)=\text{Im}_{(\mathbb{R}^{op})^3,\boldsymbol{t}}^{+}(\tilde{\mathbb{M}}^*)
		  \end{aligned}
		\end{equation*}
		Note that $\tilde{\mathbb{M}}_{\boldsymbol{t}}=\text{Im}_{\mathbb{R}^3,\boldsymbol{t}}^{+}(\mathbb{M}) \cap \text{Ker}_{\mathbb{R}^3,\boldsymbol{t}}^{-}(\mathbb{M})$, so $\text{Ker}_{\mathbb{R}^3,\boldsymbol{t}}^{-}(\tilde{\mathbb{M}})=\tilde{\mathbb{M}}_{\boldsymbol{t}}$. 
		Thus $\text{Im}_{(\mathbb{R}^{op})^3,\boldsymbol{t}}^{+}(\tilde{\mathbb{M}}^*)=(\tilde{\mathbb{M}}_{\boldsymbol{t}})^{\perp}=0$. 
    \end{proof}
    \vspace{0.3cm}
    Based on the previous results, we know that the module $\tilde{\mathbb{M}}^*$ can be decomposed into the direct sum of block modules, which are the type of $\mathcal{B}_1,\mathcal{B}_2,\mathcal{B}_3,\mathcal{B}_5$. Then the submodule $\tilde{\mathbb{M}}$ can be decomposed into the direct sum of block modules, which are the type of $\mathcal{B}_1,\mathcal{B}_2,\mathcal{B}_3,\mathcal{B}_4$. 
	
    Thus, we have proved our main theorem. 
    \vspace{0.3cm}
    \begin{theorem}
		$\mathbb{M} \cong \underset{B:\text{block}}{\bigoplus}\mathbb{M}_B$ in which $M_B\cong \bigoplus_{i=1}^{n_B}\Bbbk_{B}$ in which $n_B$ are determined by the counting functor $\mathcal{CF}$, i.e. Corollary\ref{8}. 
    \end{theorem}
    \vspace{0.3cm}

\bibliographystyle{plain}
\bibliography{reference.bib}
\end{document}